\documentclass[11pt]{amsart}

\usepackage{amsmath,amssymb,verbatim}
\usepackage{amsthm}
\usepackage{enumerate}
\usepackage{url}
\usepackage{mathtools}
\usepackage{color}
\usepackage{multicol}

\newtheorem{theorem}{Theorem}[section]
\newtheorem{lemma}[theorem]{Lemma}
\newtheorem{corollary}[theorem]{Corollary}
\newtheorem{proposition}[theorem]{Proposition}

\newtheorem{fact}[theorem]{Fact}

\theoremstyle{definition}
\newtheorem{definition}[theorem]{Definition}

\newtheorem{remark}[theorem]{Remark}

\def\VC{\operatorname{VC}}
\def\lVC{\operatorname{\ell VC}}
\def\rVC{\operatorname{\emph{r}VC}}
\def\Stab{\operatorname{Stab}}
\def\St{\operatorname{St}}
\def\cov{\operatorname{cov}}

\def\cC{\mathcal C}
\def\cF{\mathcal F}

\def\cS{\mathcal S}

\def\Z{\mathbb Z}
\def\F{\mathbb F}

\def\T{\mathbb T}

\def\T{\textnormal{T}}
\def\U{\textnormal{U}}

\newcommand{\seq}{\subseteq}
\newcommand{\inv}{^{\text{-}1}}
\newcommand{\nv}{\text{-}}

\newcommand{\clqed}{\hfill$\dashv_{\text{\scriptsize{claim}}}$}
\newcommand{\smd}{\raisebox{.75pt}{\textrm{\scriptsize{~\!$\triangle$\!~}}}}
\newcommand{\bone}{\boldsymbol{1}}

\AtBeginDocument{%
   \def\MR#1{}
}

\title{Stabilizers and NIP arithmetic regularity}

\author{G. Conant}
\address{Department of Mathematics, Statistics, and Computer Science\\
University of Illinois Chicago
}
\email{gconant@uic.edu}

\author{C. Terry}
\address{Department of Mathematics, Statistics, and Computer Science\\
University of Illinois Chicago
}
\email{caterry@uic.edu}

\date{September 4, 2025}

\thanks{GC was partially supported by an NSF grant (DMS-2452816); CT was partially supported by an NSF CAREER Award (DMS-2115518) and a Sloan Research Fellowship.}

\begin{document}

\begin{abstract}
We give a new proof of the NIP arithmetic regularity lemma for finite groups (due to the authors and Pillay), which describes the approximate structure of ``NIP sets" in finite groups, i.e., subsets  whose  collection of left translates  has bounded VC-dimension.  Our new proof avoids  sophisticated ingredients from the model theory of NIP formulas (e.g., Borel definability and generic compact domination). The key tool is an  elaboration on an elementary lemma due to Alon, Fox, and Zhao concerning the behavior of subgroups contained in stabilizers. We adapt this lemma to arbitrary subsets of stabilizers using  technical (but elementary) maneuvers based on work of Sisask. Using another trick from Alon, Fox, and Zhao, we then give an effective proof of a related result of the first author and Pillay on finite NIP sets of bounded tripling in arbitrary groups. Along the way, we show that NIP sets satisfy a strong form of the Polynomial Bogolyubov-Ruzsa Conjecture.   
\end{abstract}

\maketitle

\section{Introduction}

\subsection{Background} Given a group $G$ and an integer $d\geq 1$, we say that a subset $A\seq G$ is \textbf{$d$-NIP} if 
if the collection of  left translates of $A$ has VC-dimension less than $d$.\footnote{VC-dimension is defined in Section \ref{sec:VC}. The initialism ``NIP" comes from the model theoretic literature; see \cite[p. 2]{CPTNIP} for details.}  
In \cite{CPTNIP}, the authors and Pillay proved the following ``structure and regularity" result for $d$-NIP subsets of finite groups. 

\begin{theorem}[Conant, Pillay, Terry \cite{CPTNIP}]\label{thm:CPTNIP}
Suppose $G$ is a finite group and $A\seq G$ is $d$-NIP. Then for any $\epsilon>0$, there is a normal subgroup $H\leq G$ of index $O_{d,\epsilon}(1)$ and a $(\delta,m)$-Bohr neighborhood $B$ in $H$ with $\delta\inv,m\leq O_{d,\epsilon}(1)$ satisfying the following properties:
\begin{enumerate}[$(i)$]
\item \textnormal{(structure)} There is a set $F\seq G$ with $|F|\leq O_{d,\epsilon}(1)$ such that
\[
|A\smd FB|<\epsilon|G|.
\]
\item \textnormal{(regularity)}There is a set $Z\seq G$ with $|Z|\leq \epsilon|G|$ such that for all $g\in G\backslash Z$, either $|gB\cap A|<\epsilon|B|$ or $|gB\backslash A|<\epsilon|B|$.
\end{enumerate}
\end{theorem}

In Theorem \ref{thm:CPTNIP}, a $(\delta,m)$-Bohr neighborhood is a special kind of algebraically structured set obtained as the preimage of the open identity neighborhood of radius $\delta$ in the $m$-dimensional real torus. See Definition \ref{def:Bohrdef} and Remark \ref{rem:Bohrdef} for further details. We also  note that the main result of \cite{CPTNIP} actually asserts a  stronger form of $(i)$; see  Remark \ref{rem:NIPAR}$(3)$. 

Theorem \ref{thm:CPTNIP} is one of several closely related tame arithmetic regularity results proved  in a relatively short period of time. This began with a paper of the second author and Wolf \cite{TeWo} on ``stable" subsets of $\F_p^n$, which compares to Green's \cite{GreenSLAG} general arithmetic regularity lemma in $\F_2^n$ in direct analogy to how the Malliaris-Shelah \cite{MaShStab} regularity lemma for stable graphs compares to Szemer\'{e}di's \cite{SzemRL}   regularity lemma for all finite graphs. An ineffective generalization of \cite{TeWo} to all finite groups was obtained shortly after by the authors and Pillay \cite{CPT} (effective  results were later obtained in \cite{TeWo2} and \cite{CoQSAR}).

After the first results on stable sets in \cite{TeWo} and \cite{CPT}, three papers on NIP sets in finite groups were written in quick succession. First, Alon, Fox, and Zhao \cite{AFZ} proved an  arithmetic regularity result for NIP sets in finite abelian groups of bounded exponent. This was followed by work of Sisask \cite{SisNIP} on NIP sets in arbitrary finite abelian groups, and then finally Theorem \ref{thm:CPTNIP} above. 
None of these  results entirely subsumes the other due to various features arising from the different settings. The most significant tension  is that the results in \cite{AFZ} and \cite{SisNIP} for  abelian groups provide explicit and highly efficient bounds, whereas the bounds in Theorem \ref{thm:CPTNIP} are  ineffective due to the use of model-theoretic methods and an ultraproduct construction. Some progress on this issue was made in \cite{CoBogo} where the first author gave a ``99\% effective" generalization of Alon, Fox, Zhao \cite{AFZ} to arbitrary finite groups of bounded exponent (this will be  further explained in Section \ref{sec:methods}).

\subsection{Overview}
Our primary motivation is to further the progress toward an effective proof of Theorem \ref{thm:CPTNIP}. In this pursuit, we prove two main results, which we describe informally below (precise statements appear in the main body of the paper). 

\subsection*{Theorem \ref{thm:NIPAR}} We provide a new proof of Theorem \ref{thm:CPTNIP}. Although our bounds are still ineffective, we significantly simplify the model-theoretic machinery used in \cite{CPTNIP}, which was largely developed in a prequel paper \cite{CPpfNIP} by the first author and Pillay on NIP formulas in pseudofinite groups. That paper culminated in a ``generic compact domination" result which, as shown in \cite{CPTNIP}, corresponds precisely to NIP arithmetic regularity. Generic compact domination originated in  work of Hrushovski, Peterzil, and Pillay \cite{HPP} on NIP theories (which was motivated by conjectures about groups definable in o-minimal theories). In order to make this concept meaningful for NIP formulas, the work in \cite{CPpfNIP} used two  results of Simon \cite{SimRC,SimGCD} requiring sophisticated machinery from model theory. Our proof will not require any of these tools. In fact, the only nontrivial result about VC-dimension that we will need is Haussler's Packing Lemma (Lemma \ref{lem:HPL}). Moreover, the only use of model theory in our proof will be hidden in an application of the noncommutative version of Bogolyubov's Lemma (discussed below). This will also be the only source of ineffectiveness in our proof.  

\subsection*{Theorem \ref{thm:trip}} We give a new proof a result of the first author and Pillay \cite{CPfrNIP}, which adapts Theorem \ref{thm:CPTNIP} to the setting of finite NIP sets of bounded tripling in infinite groups. The methods in \cite{CPfrNIP} are based on the same connection to generic compact domination discussed above, but with additional complexity arising from the need to work in a locally compact setting. Our new proof again avoids these tools, and instead pushes all uses of model theory into an application of Breuillard-Green-Tao \cite{BGT} (which replaces the noncommutative Bogolyubov's Lemma used in Theorem \ref{thm:NIPAR}).  Most importantly, whereas the results in \cite{CPfrNIP} are ineffective, here we obtain bounds with a \emph{polynomial} dependence on the tripling constant and the error parameter. This improvement relies crucially on generalizations of ideas of Alon, Fox, and Zhao \cite{AFZ}, which we discuss below. Along the way, we prove a strong form of the Polynomial Bogolyubov-Ruzsa Conjecture for the special case of $d$-NIP sets (see Corollary \ref{cor:PBR}). When restricted to finite groups, Theorem \ref{thm:trip} also establishes a  variation of Theorem \ref{thm:NIPAR} with polynomial bounds in $1/\epsilon$, but with  Bohr neighborhoods replaced   by more complicated objects called ``coset nilprogressions" (see Corollary \ref{cor:CPTeff}). 

In Section \ref{sec:abelian}, we will refine Theorem \ref{thm:trip} in the special case of abelian groups.  We will also revisit the bounded exponent analogue of Theorem \ref{thm:trip} in Section \ref{sec:BE}.

\subsection{Methods}\label{sec:methods}
The heart of our arguments is a  strategy  based on the work of Alon, Fox, and Zhao \cite{AFZ} on NIP sets in finite abelian groups of bounded exponent. The general setting of this strategy is as follows.  
Let $G$ be any finite group, and fix $A\seq G$. Given $\epsilon>0$, define the ``stabilizer" $S_\epsilon=\{x\in G:|Ax\smd A|\leq\epsilon|G|\}$ (i.e., $S_\epsilon=\Stab^r_{\epsilon|G|}(A)$ in the notation of  Definition \ref{def:Stab}). The strategy now consists of three main ingredients. The first   is the following key insight from \cite{AFZ}, which says that the set $A$ can be well approximated by a union of cosets of any subgroup contained in $S_\epsilon$.\medskip

\noindent\textbf{Stabilizer Lemma.} Suppose $H$ is a subgroup of $G$ contained in $S_\epsilon$. Then there is some $F\seq G$ such that $|A\smd FH|\leq \epsilon|G|$.\medskip

\noindent This lemma yields a structure statement for $A$ in terms of $H$. A suitable regularity statement follows from the proof, but this is not made explicit in \cite{AFZ} (see Lemma \ref{lem:AFZH}, where we revisit this result in a more general setting). 

The second ingredient is the fact that the stabilizer $S_{\epsilon}$ is dense in $G$ when the set $A$ is $d$-NIP.  In particular, if $A$ is $d$-NIP then a result from VC-theory called Haussler's Packing Lemma immediately implies $|S_\epsilon|\geq(\epsilon/30)^d|G|$. This is another key observation in \cite{AFZ}, although the connection between NIP and large stabilizers is central in model theory as well (e.g., \cite[Lemma 6.3]{HPP}, \cite[Proposition 3.2]{CPpfNIP}; see also Remark \ref{rem:HPL}). 

The third ingredient is the fact that $S_\epsilon$ contains large subgroups.  For this, we move to the full setting of \cite{AFZ} where $G$ is abelian of  exponent $r$. In this case, a result of Ruzsa \cite{RuzBE} (typically called Bogolyubov's Lemma) says that if $S\seq G$ is nonempty, then the sumset $2S-2S$ contains a subgroup $H$ of index depending only on $r$ and $|G|/|S|$. Applying this to $S=S_{\epsilon/4}$, and importing the lower bound on $|S_{\epsilon/4}|$  from Haussler, we obtain $H\seq 4S_{\epsilon/4}\seq S_\epsilon$ of index $O_{d,r,\epsilon}(1)$. With the Stabilizer Lemma, this altogether yields Theorem \ref{thm:CPTNIP} for finite abelian groups of bounded exponent, but with the Bohr neighborhood $B$ \emph{equal} to the subgroup $H$ (as noted in \cite[Section 5]{AFZ}, this extra feature is not possible without the bound on the exponent).

Further, Alon, Fox, and Zhao  use a clever trick to obtain a polynomial bound in $1/\epsilon$. This requires the Bogolyubov-Ruzsa Lemma, which  is only known to hold with quasi-polynomial bounds. However, their argument is delicately tailored so that this does not  affect the overall polynomial dependence on $\epsilon$. In \cite{CoBogo}, the first author used Hrushovski's \cite{HruAG} non-commutative analogue of Bogolyubov-Ruzsa for groups of bounded exponent  to execute the same trick in the nonabelian case.

We can now summarize the main ideas of our work. The proof of the Stabilizer Lemma in \cite{AFZ} is  short and elementary, but very much relies on the partition structure coming from cosets of $H$. In Section \ref{sec:lemmas}, we will prove a suitable adaptation  applicable to arbitrary subsets of stabilizers (see Lemma \ref{lem:structure}). Here our arguments borrow heavily from Sisask's \cite{SisNIP} results on NIP sets in finite abelian groups. So while the work in Section \ref{sec:lemmas} represents the main technical obstacle required for our results, we stress that several key ideas are already present in Sisask's work, albeit embedded in  the abelian setting and  Fourier analytic techniques therein. Lemma \ref{lem:structure} also requires some subtle care that only arises in  nonabelian groups. 

With Section \ref{sec:lemmas} in hand, we can then approach our main results using the same Alon-Fox-Zhao strategy of finding well-structured  sets inside of stabilizers. 
For Theorem \ref{thm:NIPAR}, the key tool  is a noncommutative version of Bogolyubov's Lemma for arbitrary finite groups, proved by the first author in \cite{CoBogo}. However, there is no known effective proof of this result, which is the only reason our bounds in Theorem \ref{thm:NIPAR} remain ineffective. For Theorem \ref{thm:trip}, we combine a similar stabilizer strategy with  a generalization of the trick from \cite{AFZ} alluded to above (see Lemma \ref{lem:AFZ}). But this requires a noncommutative version of the Bogolyubov-Ruzsa Lemma, which is provided by the  Breuillard-Green-Tao \cite{BGT} structure theorem for approximate groups.

\subsection*{Outline} Section \ref{sec:pre} contains all of the preliminaries needed for our proofs. This includes background on VC-dimension and Haussler's Packing Lemma, as well as details on several ``Bogolyubov-Ruzsa-type" results from arithmetic combinatorics. In Section \ref{sec:lemmas}, we prove the  technical lemmas on stabilizers mentioned above. We then prove Theorem \ref{thm:NIPAR} in Section \ref{sec:ineff}, and compare and contrast our work with  \cite{CPTNIP}. Section \ref{sec:eff} contains the proof of Theorem \ref{thm:trip} and the related results discussed above. Finally, in Section \ref{sec:BE}, we revisit the bounded exponent case and prove some additional results. 

\subsection*{Acknowledgments} The authors thank Anand Pillay and Julia Wolf for their helpful comments on a preliminary draft of this paper. Thanks also to Tom Sanders for pointing us to the work in \cite{LovReg}.

\section{Preliminaries}\label{sec:pre}

Given the length of this section, we note that our first main result (Theorem \ref{thm:NIPAR}) only requires the preliminaries in Subsections \ref{sec:not}, \ref{sec:VC}, and \ref{sec:BL}. The remaining material will not be needed until Section \ref{sec:eff}, where we prove the second main result (Theorem \ref{thm:trip}) and related applications.

\subsection{Notation and basic definitions.}\label{sec:not}
Throughout the paper, $\log$ and $\exp$ denote the base $2$ logarithm and exponential. We will restrict the variable $\epsilon$ to the interval $(0,1)$, regardless of whether results are still true for larger values. (This is done to avoid irrelevant calculations.)

Let $G$ be a group. Given $A,B\seq G$, we let  $AB=\{ab:a\in A,~b\in B\}$ and $A\inv=\{a\inv:a\in A\}$. For $n\geq 1$, we inductively define $A^n$ by setting $A^1=A$ and $A^{n+1}=A^nA$. Following the conventions of \cite[Definition 2.1$(i)$]{BGT}, we call a set $A\seq G$ \textbf{symmetric} if $A=A\inv$ \emph{and} $A$ contains the identity of $G$.

When $G$ is abelian, we will switch to additive notation. For example, we write $\nv A$ rather than $A\inv$, $nA$ rather than $A^n$, $A+B$ rather than $AB$, etc.

\begin{definition}[covering bound]
Given nonempty sets $A,B\seq G$ and a real number $N\geq 1$, we write $\cov(A:B)\leq N$ to mean that $A\seq FB$ for some $F\seq A$ with $|F|\leq N$.
\end{definition}

The next result is a standard exercise (see \cite[Lemma 5.1]{BGT}\footnote{This reference inadvertently omits the necessary assumption that $B$ is symmetric. The result holds without symmetry if $B^2$ is replaced by $BB\inv$; see \cite[Lemma 3.6]{TaoPSE}.}).

\begin{lemma}[Ruzsa's Covering Lemma]\label{lem:RCL}
Suppose $A,B\seq G$ are  finite sets with $B$ symmetric. Then $\cov(A:B^2)\leq |AB|/|B|$.
\end{lemma}

Now we define left and right stabilizers of finite sets in groups.

\begin{definition}\label{def:Stab}
Let $A\seq G$ be finite. Given a real number $N>0$, define
\begin{align*}
\Stab^\ell_{N}(A) &= \{x\in G:|xA\smd A|\leq N\},\text{ and}\\
\Stab^r_{N}(A) &= \{x\in G:|Ax\smd A|\leq N\}.
\end{align*}
We will often set $N=\epsilon|A|$ for some $\epsilon\in(0,1)$. Thus for the sake of brevity, given $\epsilon\in(0,1)$ and $\bullet\in \{r,\ell\}$, we let
\[
\St^\bullet_\epsilon(A)=\Stab^\bullet_{\epsilon|A|}(A).
\]
\end{definition}

\begin{proposition}\label{prop:Stab}
Fix a finite set $A\seq G$ and $M,N>0$.
\begin{enumerate}[$(a)$]
\item $\Stab^\bullet_{N}(A)$ is symmetric (where $\bullet\in\{r,\ell\}$).
\item For any  $M>0$, $\Stab^\bullet_{M}(A)\Stab^\bullet_{N}(A)\seq \Stab^\bullet_{M+N}(A)$ (where $\bullet\in\{r,\ell\}$).
\item $\Stab^\ell_N(A)=\Stab^r_N(A\inv)$. Thus if  $\epsilon\in(0,1)$ then  $\St^\ell_\epsilon(A)=\St^r_\epsilon(A\inv)$.
\item If $N<2|A|$ then $\Stab^\ell_{N}(A)\seq AA\inv$ and $\Stab^r_{N}(A)\seq A\inv A$. Thus if $A\neq\emptyset$ and $\epsilon\in(0,1)$ then $\St^\ell_\epsilon(A)\seq AA\inv$ and $\St^r_\epsilon(A)\seq A\inv A$.
\end{enumerate}
\end{proposition}
\begin{proof}
Part $(a)$ follows from the fact that for any $x\in G$, $|xA\smd A|=|A\smd x\inv A|$.  Part $(b)$ follows from the fact that  for any $x,y\in G$, 
$$
|xyA\smd A|=|yA\smd x\inv A|\leq  |yA\smd A|+|A\smd x\inv A|= |yA\smd A|+|xA\smd A|.
$$
Part $(c)$ follows from part $(a)$ and the fact that for any $x\in G$, $|xA\smd A|=|A\inv x\inv\smd A\inv |$ (since $(xA\smd A)\inv=A\inv x\inv\smd A\inv$). 

For part $(d)$, first suppose $x\in \Stab^\ell_N(A)$. Then $|xA\smd A|\leq N<2|A|$, which implies $xA\cap A\neq\emptyset$ (otherwise $|xA\smd A|=|xA\cup A|=2|A|$). So there are $a,b\in A$ such that $a=xb$, i.e., $x\in AA\inv$. A similar argument shows $\Stab^r_N(A)\seq A\inv A$.
\end{proof}

\subsection{VC-dimension in groups}\label{sec:VC}

Let $X$ be a set. A \textbf{set system on $X$} is a family $\cF$ of subsets of $X$. We say that a set system $\cF$ (on $X$) \textbf{shatters} a subset $A\seq X$ if $\mathcal{P}(A)=\{A\cap S:S\in\cF\}$. The \textbf{VC-dimension of $\cF$}, denoted $\VC(\cF)$, is the maximum cardinality of a finite subset of $X$ shattered by $\cF$ (or $\VC(\cF)=\infty$ if $\cF$ shatters arbitrarily large finite subsets of $X$). 

\begin{lemma}[Haussler's Packing Lemma \cite{HaussPL}]\label{lem:HPL}
Let $\cF$ be a set system on a finite set $X$ with $\VC(\cF)=d$. Fix $\epsilon\in (0,1)$ and suppose $\cS\seq\cF$ is such that $|A\smd B|>\epsilon|X|$ for all distinct $A,B\in\cS$. Then $|\cS|\leq (30/\epsilon)^d$.\footnote{The bound in \cite{HaussPL} is actually $e(d+1)(2e/\epsilon)^d$, which is less than $(30/\epsilon)^d$ assuming $d\geq 1$. On the other hand, note that if $d=0$ then $|\cF|\leq 1$.}
\end{lemma}

\begin{remark}\label{rem:HPL}
Haussler's Packing Lemma is the \emph{only} fact about VC-dimension needed for the proofs of our main results (via Proposition \ref{prop:HPL} below). This is worth emphasizing in light of the heavy machinery from model theory used in the proof of Theorem \ref{thm:CPTNIP} (recall the discussion of Theorem \ref{thm:NIPAR} in the introduction). That being said, the proof of Haussler's result is rather complicated. So it is also worth noting that one can obtain this lemma with slightly weaker bounds using other means. For example, Lovasz and Szegedy \cite{LovSzeg} give a short proof with the bound $(80d/\epsilon^{20})^d$ using only the VC-Theorem and the Sauer-Shelah Lemma. In fact, a bound of the form $O_d((1/\epsilon)^{(1+o_d(1))d})$ can be obtained just from the Sauer-Shelah Lemma (which has an elementary proof). This argument is sketched in the discussion after \cite[Theorem 2.1]{Moran-Yeh} (see also \cite[Remark 5.21]{CGH2}). 
\end{remark}

We now define some specific set systems in groups.

\begin{definition}
Let $G$ be a group and fix  subsets $A,B\seq G$. 
\begin{enumerate}[$(1)$]
\item Define $\cF^\ell_B(A)=\{xA:x\in B\}$ and $\cF^r_B(A)=\{Ax:x\in B\}$. (Observe that $\cF_B^{\ell}(A)$ and $\cF_B^r(A)$ can be viewed as set systems on $BA$ and $AB$, respectively.)
\item Given $\bullet\in\{\ell,r\}$, set $\VC^\bullet_B(A)=\VC(\cF^\bullet_B(A))$.
\end{enumerate}
\end{definition}

The following is a basic exercise.

\begin{proposition}\label{prop:VCinv}
For any group $G$ and $A,B\seq G$,  $\VC^\ell_B(A)=\VC^r_{B\inv}(A\inv)$.
\end{proposition}

\begin{remark}\label{rem:VCvariants}
For our main results, the relevant dimensions associated to a single set $A\seq G$ will be $\VC_A(A)$ and $\VC_{A\inv}(A)$. This differs from many other sources  (e.g., \cite{AFZ,CoBogo,CoQSAR,CPfrNIP,CPTNIP}), which focus on $\VC^\ell_G(A)$ and/or $\VC^r_G(A)$. In \cite{SisNIP}, Sisask defines yet another variation that we denote $\dim_{\bullet\!\VC}(A)$
for $\bullet\in\{\ell,r\}$ (see Definition \ref{def:SVC}). In the appendix, we will examine the relationships between these various notions. 
The brief summary is that one can establish uniform bounds (which are at worst double-exponential) between any two values from the following set: 
\[
\{\VC^\bullet_G(A),~\VC^\bullet_{A\inv}(A),~\dim_{\bullet\!\VC}(A):\bullet\in\{r,\ell\}\}.
\]
However, while $\VC^\bullet_A(A)\leq \VC^\bullet_G(A)$, we have been unable to determine whether $\VC^\bullet_G(A)$ (or any value in the above set) can be uniformly bounded above  by some function of $\VC^\bullet_A(A)$. We have not even been able to find a uniform comparison between $\VC^\ell_A(A)$ and $\VC^r_A(A)$.
\end{remark}

Next we show that in the presence of bounded VC-dimension, stabilizers are  large. This is a direct consequence of Haussler's Packing Lemma, and the argument is essentially the same as that used by Alon, Fox, and Zhao \cite[Lemma 2.2]{AFZ}  (see also  \cite[Corollary 2.7$(b)$]{CoQSAR}). We have formulated the result to include some further  generality and to explicitly state additional features that arise from the proof.

\begin{proposition}\label{prop:HPL}
Let $G$ be a  group. Fix nonempty finite sets  $A,B\seq G$, a real number $N>0$, and some $\epsilon\in (0,1)$.
\begin{enumerate}[$(a)$]
\item If $d=\VC^\ell_B(A)$ and $\epsilon\leq N/|BA|$ then 
\[
\cov(B:\Stab^\ell_{N}(A)\cap B\inv B)\leq (30/\epsilon)^{d}.
\]
\item If $d=\VC^r_B(A)$ and $\epsilon\leq N/|AB|$ then 
\[
\cov(B\inv:\Stab^r_{N}(A)\cap BB\inv)\leq (30/\epsilon)^{d}.
\]
\end{enumerate}
\end{proposition}
\begin{proof}
We prove part $(a)$. Part $(b)$ can then be obtained via a similar argument, or by applying  $(a)$ to $B\inv$ and $A\inv$. 

Recall we can view $\cF^\ell_B(A)$ as a set system on the finite set $BA$. Call $E\seq B$ \emph{separated} if $|xA\smd yA|>N$ for all distinct $x,y\in E$. Let $E\seq B$ be a separated set of maximal size. Since $N\geq\epsilon|BA|$, we have $|E|\leq (30/\epsilon)^{d}$ by Haussler's Packing Lemma. Now fix $x\in B$. By maximality, there is some $y\in E$ such that $|xA\smd yA|\leq N$, i.e., $|y\inv xA\smd A|\leq N$, i.e., $y\inv x\in \Stab^\ell_{N}(A)$. Note also that $y\inv x\in B\inv B$. Hence $x\in y(\Stab^\ell_{N}(A)\cap B\inv B)$. This shows $B\seq E(\Stab^\ell_{N}(A)\cap B\inv B)$. 
\end{proof}

As implied by the previous discussion, for our results we will only need the following weaker formulation  of a special case of Proposition \ref{prop:HPL}  (which uses the $\St^\bullet_\epsilon(A)$ notation from Definition \ref{def:Stab}).

\begin{corollary}\label{cor:HPL}
Let $G$ be a group and fix nonempty finite sets  $A,B\seq G$.
\begin{enumerate}[$(a)$]
\item If $\VC^\ell_B(A)\leq d$ and $\epsilon\in (0,1)$,  then $\cov(B:\St^\ell_\epsilon(A))\leq (30|BA|/\epsilon|A|)^d$.
\item If $\VC^r_B(A)\leq d$ and $\epsilon\in (0,1)$, then $\cov(B\inv:\St^r_\epsilon(A))\leq (30|AB|/\epsilon|A|)^d$.
\end{enumerate}
\end{corollary}

\subsection{Bogolyubov's Lemma in finite groups}\label{sec:BL}

We first recall a suitable notion of Bohr neighborhoods in noncommutative groups. In the following definition, $\U(m)$ denotes the complex unitary group of degree $m$,  and $\T(m)$ denotes the subgroup of diagonal matrices (so $\T(m)$ is isomorphic to the $m$-dimensional torus $(S^1)^m$). We equip $\U(m)$ with the metric induced by the operator norm (which then restricts to the product of complex distance metric on $\T(m)$).

\begin{definition}\label{def:Bohrdef}
Let $G$ be a group. Given  $\delta\in (0,1)$ and $m,n\in \mathbb{Z}^{\geq 1}$, a \textbf{$(\delta,m,n)$-Bohr neighborhood in $G$} is a subset $B\seq G$  of the form $B=\tau\inv(U\cap K)$ where:
\begin{enumerate}[\hspace{5pt}$\ast$]
\item $\tau\colon G\to \U(m)$ is a group homomorphism, 
\item $U$ is the open identity  neighborhood in $\U(m)$ of radius $\delta$, and
\item $K$ is a normal subgroup of $\tau(G)$ with $[\tau(G):K]=n$ and $K\seq\T(m)$. 
\end{enumerate}
We further say that $B$ has \textbf{complexity $c$} if $\max\{\delta\inv,m,n\}\leq c$.
\end{definition}

\begin{remark}\label{rem:Bohrdef}
In the commutative setting, one typically only considers $(\delta,m,1)$-Bohr neighborhoods (called $(\delta,m)$-Bohr neighborhoods in \cite{CPTNIP}\footnote{This is slightly inaccurate since in \cite{CPTNIP}, $\T(n)$ is given the product of the arclength metric on $S^1$ (rather than the complex distance metric). Thus  a $(\delta,m,1)$-Bohr neighborhood here is actually a $(\delta',m)$-Bohr neighborhood in \cite{CPT}, where $\delta'$ depends uniformly only $\delta$.}). In this case, the homomorphism $\tau$ maps to $\T(m)$, and so the ambient unitary group plays no role. On the other hand, the results in \cite{CPTNIP} for nonabelian groups involve $(\delta,m,1)$-Bohr neighborhoods in a normal subgroup $H$ of index $n$. The previous definition captures this since a $(\delta,m,n)$-Bohr neighborhood in $G$ is a $(\delta,m,1)$-Bohr neighborhood the normal subgroup $H=\tau\inv(K)$, which has index at most $n$. However, it is not explicitly evident from the setup in \cite{CPTNIP} that the map from $H$ to $\T(m)$  is the restriction of a unitary representation of $G$. For this reason, the above definition of $(\delta,m,n)$-Bohr neighborhood was formulated in \cite[Definition 4.3]{CP-AVSAR}. 
\end{remark}

Next we recall some basic properties of Bohr neighborhoods.

\begin{fact}\label{fact:Bohr}
Let $B$ be a $(\delta,m,n)$-Bohr neighborhood in a  group $G$.
\begin{enumerate}[$(a)$]
\item $B$ is symmetric.
\item There is a $(\delta/2,m,n)$-Bohr neighborhood $C$ in $G$ such that $C^2\seq B$.
\item $\cov(G:B)\leq n\lceil 2\pi/\delta\rceil^m$. 
\end{enumerate}
\end{fact}
\begin{proof}
Parts $(a)$ and $(b)$ are clear. Part $(c)$ is well-known when $G$ is finite (see \cite[Lemma 4.1]{GreenSLAG}, \cite[Lemma 4.20]{TaoVu}, \cite[Propsition 4.5]{CPTNIP}). An elementary proof for general $G$ follows from \cite[Lemma 5.5]{CHP} (see  \cite[Remark 4.5(3)]{CP-AVSAR}).
\end{proof}

We can now state the noncommutative version of Bogolyubov's Lemma. For simplicity, we include a symmetry assumption, which is not made in the sources discussed below. 

\begin{theorem}\label{thm:Bogo}
Let $G$ be a finite group and fix a symmetric set $S\seq G$ with $|S|\geq\epsilon|G|$. Then $S^4$ contains a Bohr neighborhood of complexity $O_\epsilon(1)$.
\end{theorem}

At the generality of arbitrary finite groups, this result was first proved by the first author in \cite[Theorem 1.2]{CoBogo} using work of Sanders \cite{SanBS} (but with the slightly weaker notion of Bohr neighborhood discussed in Remark \ref{rem:Bohrdef}). The name \emph{Bogolyubov's Lemma} comes from work of Ruzsa \cite{Ruz94}, who established the result for finite abelian groups (with explicit bounds) using ideas of Bogolyubov \cite{Bog39}. More recently, Theorem \ref{thm:Bogo} was generalized to arbitrary amenable groups by the first author, Hrushovski, and Pillay  \cite[Theorem 5.9]{CHP} using Hrushovski's Stabilizer Theorem \cite{HruAG}, and again by first author and Pillay \cite[Theorem 5.1]{CP-AVSAR} using arithmetic regularity  for ``stable functions" on groups. These later proofs use the definition of Bohr neighborhood given above (see \cite[Proposition 4.4]{CP-AVSAR} and surrounding remarks).

All existing proofs of Theorem \ref{thm:Bogo} require model-theoretic tools at some level. That said, this result will allow us to avoid all of the model-theoretic machinery around NIP formulas  used in \cite{CPTNIP} to prove Theorem \ref{thm:CPTNIP} (as described in the introduction). So it is worth emphasizing that the first author's original proof of Theorem \ref{thm:Bogo} in \cite{CoBogo} was heavily inspired by various techniques developed in \cite{CPTNIP} for working with Bohr neighborhoods in ultraproducts, as well as the earlier work of Pillay \cite{PiRCP} on compactifications of pseudofinite groups (which also played a key role in \cite{CPTNIP}).

\subsection{Pl\"{u}nnecke-Ruzsa and tupling parameters}

This section contains some basic preliminaries related to the notions of bounded ``doubling" and ``tripling". Throughout this section, we let $G$ be a group. 

\begin{definition}\label{def:TP}
Given a nonempty finite set $A\seq G$, define the values 

\begin{minipage}{.4\textwidth}
    \begin{align*}
    \sigma[A] &= |A^2|/|A|\\
    \tau[A] &= |A^3|/|A|
    \end{align*}
\end{minipage}%
\begin{minipage}{.5\textwidth}
    \begin{align*}
    \delta[A] &= |AA\inv|/|A|\\
    \alpha[A] &=|AA\inv A|/|A|
    \end{align*}
\end{minipage}
\end{definition}

We  refer to the above values as ``tupling parameters" associated to $A$. Of these, $\sigma[A]$ and $\delta[A]$ are standard (at least in the abelian context, see \cite[Definition 2.4]{TaoVu}). The value $\delta[A]$ also corresponds to $\exp(d(A,A))$ where $d$ denotes Ruzsa distance \cite[Definition 3.1]{TaoPSE}. 

We will use the following fundamental result.  See \cite[Lemma 3.4]{TaoPSE} for a proof of part $(a)$ and a discussion of the history. A proof of part $(b)$ can be found in \cite[Corollary 6.29]{TaoVu}.

\begin{proposition}[Pl\"{u}nnecke-Ruzsa Inequalities]\label{prop:PRI}
Suppose $A\seq G$ is finite and nonempty.
\begin{enumerate}[$(a)$]
\item For any $\epsilon_1,\ldots,\epsilon_n\in\{1,\nv 1\}$, $|A^{\epsilon_1}\cdots A^{\epsilon_n}|\leq \tau[A]^{O_n(1)}|A|$.
\item If $G$ is abelian then, for any $m,n\geq 1$, $|mA-nA|\leq \sigma[A]^{m+n}|A|$.
\end{enumerate}
\end{proposition}

 The next fact includes some  specific inequalities along the same lines.

\begin{fact}\label{fact:triplings} Fix a nonempty finite set $A\seq G$. 
\begin{enumerate}[$(a)$]
\item $\sigma[A]\leq\tau[A]$ and $\delta[A]\leq\alpha[A]$.
\item $\delta[A]\leq\sigma[A]^2$.
\item $\tau[A]\leq\alpha[A]\sigma[A]^2$ and $\alpha[A]\leq\tau[A]\sigma[A]^2$.
\item $\alpha[A]$ and $\sigma[A]$ cannot be bounded uniformly in terms of each other.
\item If $G$ is abelian then $\tau[A]\leq \delta[A]^3$, and hence $\sigma[A]$, $\tau[A]$, $\delta[A]$, and $\alpha[A]$ are all bounded uniformly in terms of each other. 
\end{enumerate}
\end{fact}
\begin{proof}
Part $(a)$ is trivial. Part $(b)$ follows from the Ruzsa triangle inequality  \cite[Lemma 3.2]{TaoPSE}. Part $(c)$ is  an exercise involving the Ruzsa triangle inequality, which we leave to the reader.\footnote{Part $(c)$ will not be required for any of our proofs, and instead will only be used to provide context for various assumptions in the statements of our results.}  For part $(d)$, see \cite[Remark 2.2]{CoBogo}. For part $(e)$, see \cite[Corollary 6.28]{TaoVu}.
\end{proof}

In the next proposition, we note that for certain set systems in groups, VC-dimension $0$ corresponds to a certain tupling parameter of $1$, and thus is characterized by strong algebraic structure. This will be convenient later for technical reasons, and also draws a nice connection to basic known facts  on doubling. The main content of part $(b)$ was  first shown by Sisask \cite[Proposition 4.7]{SisNIP} (modulo Proposition \ref{prop:VCvariants}$(d)$ and the fact that a set system has VC-dimension $0$ if and only if its dual does).

\begin{proposition}\label{prop:d=0} Let $G$ be a group and fix a nonempty finite set $A\seq G$.
\begin{enumerate}[$(a)$]
\item Suppose $B\seq G$ is nonempty and finite. Then $\VC^\ell_B(A)=0$ if and only if $|BA|=|A|$, and $\VC^r_B(A)=0$ if and only if $|AB|=|A|$.
\item The following are equivalent.
\begin{enumerate}[$(i)$]
\item $\VC^\ell_{A\inv}(A)=0$
\item $\VC^r_{A\inv}(A)=0$.
\item $|AA\inv|=|A|$.
\item $|A\inv A|=|A|$.
\item $A$ is a coset of a subgroup of $G$.
\end{enumerate}
\item The following are equivalent.
\begin{enumerate}[$(i)$]
\item $\VC^\ell_{A}(A)=0$
\item $\VC^r_A(A)=0$.
\item $|AA|=|A|$.
\item $A$ is a coset $aH$ of a subgroup $H\leq G$ with $aH=Ha$.
\end{enumerate}
\end{enumerate}
\end{proposition}
\begin{proof}
Part $(a)$. First note that for any nonempty set system $\cF$, $\VC(\cF)=0$ if and only if $|\cF|=1$. Thus we have $\VC^\ell_B(A)=0$ if and only if $xA=yA$ for all $x,y\in B$. So $\VC^\ell_B(A)=0$ if and only if $BA=bA$ for any fixed $b\in B$. This latter condition is clearly equivalent to $|BA|=|A|$. The argument for $\VC^r_B(A)=0$ is similar.

Part $(b)$. We have $(i)\Leftrightarrow (iv)$ and $(ii)\Leftrightarrow (iii)$ by part $(a)$. Moreover, $(v)\Rightarrow (iv)$ and $(v)\Rightarrow (iii)$ are easy to verify. Finally, $(iii)\Rightarrow (v)$ and $(iv)\Rightarrow (v)$ are standard exercises. For example, assume $|AA\inv|=|A|$ and set $H=Aa\inv$ where $a\in A$ is some fixed element.  Then $HH\inv=AA\inv$ so $|HH\inv|=|A|=|H|$. As in part $(a)$, this implies $HH\inv =Hx\inv$ for any fixed $x\in H$. Since $1\in H$, we have $HH\inv =H$. This shows that $H$ is a subgroup. Hence $A=Ha$ is a coset. The argument for $(iv)\Rightarrow (v)$ is similar.

Part $(c)$. The equivalence of $(i)$, $(ii)$, and $(iii)$ follows from part $(a)$. The equivalence of $(iii)$ and $(iv)$ is again a well-known basic exercise (see, e.g., \cite[Proposition 1.6]{BreuH5P}). 
\end{proof}

\subsection{Breuillard, Green, and Tao}
 In this section, we describe a result of Breuillard, Green, and Tao \cite{BGT}, which can be viewed as an analogue of Theorem \ref{thm:Bogo} in which dense sets in finite groups are replaced by  finite sets with bounded tripling in  arbitrary groups. In this case, rather than a Bohr neighborhood, the key structural ingredient is a special kind of finite  set called a ``coset nilprogression", which is a noncommutative analogue of a generalized arithmetic progression in an abelian group. The full definition is somewhat lengthy, and the finer details will not be needed here. So we refer the reader to Definitions 2.3, 2.5, and 2.6 in \cite{BGT}, which altogether define the notion of a \emph{coset nilprogression $P$ of rank $r$, step $s$, and in $t$-normal form}. We say that $P$ has \emph{complexity $c$} if $\max\{r,s,t\}\leq c$. The next fact lists the only specific properties of coset nilprogressions that we will need.

\begin{fact}\label{fact:coverQ}
Let $G$ be a group and let $P$ be a coset nilprogression in $G$.
\begin{enumerate}[$(a)$]
\item $P$ is symmetric.
\item If $P$ has complexity $c$, then there is a finite symmetric set $Q$ (in fact, a coset nilprogression)  such that $Q^2\seq P$ and $|P|\leq O_c(|Q|)$. 
\end{enumerate}
\end{fact}
\begin{proof}
Part $(a)$ is immediate from the definitions in \cite{BGT}. Part $(b)$ is a consequence of \cite[Lemma C.1]{BGT}.
\end{proof}

Note that Fact \ref{fact:coverQ}$(b)$ is analogous to the property of Bohr neighborhoods given in Fact \ref{fact:Bohr}$(b)$. 

Next we state Breuillard, Green, and Tao's result.

\begin{theorem}[\cite{BGT,TaoPSE}]\label{thm:BGT}
Let $G$ be a group and fix a  finite symmetric set $S\seq G$ with $|S^3|\leq k|S|$. Then there is some integer $n= O_k(1)$ and a coset nilprogression $P$  of rank and step $O(\log(2k))$, and in $O_k(1)$-normal form, such that $P\seq S^{n}$  and $|S|\leq O_k(|P|)$.
\end{theorem}

The previous result follows from  \cite[Theorem 2.12]{BGT} which makes the (qualitatively) stronger assumption that $S$ is a $k$-approximate group. To obtain the version above, one uses the fact that if   $|S^3|\leq k|S|$ then $S^3$ is an $O(k^{O(1)})$-approximate group  (see \cite[Corollary 3.10]{TaoPSE}). It is also important to note that one can obtain an absolute constant value for $n$ at the cost of an ineffective $O_k(1)$ bound on the rank and step of $P$ (see \cite[Theorem 2.10]{BGT}). However, the nature of our results makes the above version more useful. For this reason, we make the following definition for the sake of convenience in later arguments.

\begin{definition}\label{def:BGTn}
Given $k\geq 1$, let $n(k)=\max\{n,3\}$ where $n= O_k(1)$ is as in Theorem \ref{thm:BGT}.
\end{definition}

\subsection{The Bogolyubov-Ruzsa Lemma}\label{sec:BRL}

We next state the \emph{Bogolyubov-Ruzsa Lemma}, which can be viewed as a commutative analogue of Theorem \ref{thm:BGT} (or, more accurately, of the related result \cite[Theorem 2.10]{BGT} mentioned after Theorem \ref{thm:BGT}). The first statement of this kind was proved by Ruzsa \cite{Ruz94} for $\Z$, and later generalized to arbitrary abelian by Green and Ruzsa \cite{GrRuz}. The following version gives the best-known bounds, due to Sanders \cite{SanBR}, which are quasi-polynomial in the doubling constant (we have written the bounds in a slightly weaker form for the sake of simplicity). 

\begin{theorem}[\cite{SanBR}]\label{thm:BRL}
Let $G$ be an abelian group and fix a nonempty finite set $S\seq G$ with $|2S|\leq k|S|$. Then there is a proper coset progression $P$ in $G$ of rank $O((\log 2k)^{6})$ such that $P\seq 2S-2S$ and $|S|\leq \exp(O((\log 2k)^{7})|P|$.
\end{theorem}

As with nilprogressions, we refer the reader to \cite{BGT} for the definition of a \emph{proper coset progression}. See also \cite{GrRuz,SanBR,TaoVu}.\footnote{As in \cite{BGT}, a coset progression for us is symmetric and ``centered" at the identity, which differs from some sources where translates are allowed.} We will only need to recall the following effective version of Fact \ref{fact:coverQ}$(b)$, which is a standard exercise.

\begin{fact}\label{fact:coverQa}
Let $G$ be an abelian group and let $P$ be a proper coset progression of rank $r$. Then there is a symmetric set $Q$ (in fact, a  coset progression of rank $r$) such that $2Q\seq P$ and $|P|\leq 4^r|Q|$. 
\end{fact}

The shape of the best possible bounds in Theorem \ref{thm:BRL} is an open problem of active interest.  Optimally, one would wish for a polynomial bound $k^{O(1)}$ on $|S|$ and a logarithmic bound $O(\log 2k)$ on the rank of $S$, but this is known to be false in general due to work of Lovett and Regev \cite[Theorem 1.4]{LovReg} (though a similar statement with these bounds and involving a relaxed form of coset progression is open; see \cite[Conjecture 1.2]{LovReg}). The situation is less murky in the bounded exponent case, where \emph{Polynomial Bogolyubov-Ruzsa Conjecture} states that if $G$ and $S$ are as in Theorem \ref{thm:BRL}, and $G$ has exponent $q$, then there is a subgroup $H\seq 2S-2S$ with $|S|\leq k^{O_q(1)}|H|$ (e.g., see  \cite[p. 2]{GGMT2}, or \cite{Lov-surv} for the exponent $2$ case).  A weaker version of this statement, called the Polynomial Freiman-Ruzsa Conjecture,  was recently proved  by Gowers, Green, Manners, and Tao \cite{GGMT2,GGMT}. We will connect this  to our results at the end of  Section \ref{sec:BE}.

Beyond abelian groups, this line of investigation extends to questions about optimal bounds in Theorem \ref{thm:BGT}, which currently has no known effective proof (though there are effective results in certain classes of groups, e.g., \cite{Tointon}). In Corollaries \ref{cor:PBR} and \ref{cor:PBRa}, we will prove strong versions of  Theorems \ref{thm:BGT} and \ref{thm:BRL}  with polynomial bounds for NIP sets.

\section{Main lemmas}\label{sec:lemmas}

In this section, we prove two technical lemmas on stabilizers needed for our main results. Roughly speaking, Lemma \ref{lem:regularity} is a general ``regularity statement" in terms of arbitrary subsets of stabilizers, while Lemma \ref{lem:structure} gives a general ``structure statement" in terms of such sets.  As explained in the introduction,  this lemma is a substantial elaboration on the Stabilizer Lemma of Alon, Fox, and Zhao \cite{AFZ} (which only applies to subgroups contained in stabilizers). We also emphasize that the proofs in this section are heavily based on work of Sisask \cite{SisNIP}. This is especially the case for Lemma \ref{lem:structure}, which has been assembled from several parts of \cite{SisNIP}. That said, much of \cite{SisNIP} takes place in the abelian setting, and thus the final statement of Lemma \ref{lem:structure} does not directly follow from any one result there. Indeed, for nonabelian groups we must take additional care to balance a somewhat subtle relationship between left and right stabilizers.  

Throughout this section, let $G$ be a fixed group. We first establish general notation for the ``error sets" involved in the regularity statement given by condition $(ii)$ of Theorem \ref{thm:CPTNIP}. 

\begin{definition}\label{def:Zset}
Fix finite sets $A,X\seq G$. Given $\epsilon\in (0,1)$, define 
\begin{align*}
Z^\ell_\epsilon(A,X) &= \{g\in G:\min\{|gX\cap A|,|gX\backslash A|\}\geq\epsilon|X|\},\text{ and}\\
Z^r_\epsilon(A,X) &= \{g\in G:\min\{|Xg\cap A|,|Xg\backslash A|\}\geq\epsilon|X|\}.
\end{align*}
\end{definition}

Note that condition $(ii)$ of Theorem \ref{thm:CPTNIP} can be rephrased simply as the inequality $|Z^\ell_\epsilon(A,B)|\leq \epsilon|G|$ (where $B$ is as in the theorem). For the sake of brevity, our main results in Sections \ref{sec:ineff} and \ref{sec:eff} will be written in this way. 

In the proofs below, we will tacitly use the following remark regarding finiteness of certain stabilizers and  error sets.

\begin{remark}\label{rem:Zplace}
Fix finite nonempty sets $A,X\seq G$ and $\epsilon\in (0,1)$. 
\begin{enumerate}[$(a)$]
\item If $g\in Z^\ell_\epsilon(A,X)$ then $gX\cap A\neq\emptyset$, i.e., $g\in AX\inv$.  So $Z^\ell_{\epsilon}(A,X)\subseteq AX\inv$ and, consequently, $Z^\ell_\epsilon(A,X)$ is finite. Similarly, $Z^r_\epsilon(A,X)\seq X\inv A$.
\item Recall that for $\bullet\in\{\ell,r\}$, $\St^\bullet_\epsilon(A)$ denotes $\Stab^\bullet_{\epsilon|A|}(A)$. By Proposition \ref{prop:Stab}$(d)$, $\St^\bullet_\epsilon(A)$ is finite. 
\end{enumerate}
\end{remark}

We now prove the main results of this section.   The first, Lemma \ref{lem:regularity} below, says that given any subset $X$ of the stabilizer of a set $A$, most translates of $X$ are either mostly inside $A$ or mostly disjoint from $A$. 

\begin{lemma}[regularity]\label{lem:regularity}
Fix a finite set $A\seq G$, a real number $N>0$, and some $\epsilon\in (0,1)$. Suppose $X\seq \Stab^r_N(A)$ is finite and nonempty. Then $|Z^\ell_\epsilon(A,X)|\leq 2N/\epsilon$. In particular,  if $X\seq\St^r_{\epsilon^2/2}(A)$ is finite and nonempty, then $|Z^\ell_\epsilon(A,X)|\leq \epsilon|A|$.
\end{lemma}
\begin{proof}
 Let $Z=Z^\ell_\epsilon(A,X)$. 
 Note that $Z=Z^\ell_\epsilon(G\backslash A,X)$ and $\Stab^r_{N}(A)=\Stab^r_{N}(G\backslash A)$. So, after replacing $A$ with $G\backslash A$ if necessary, we may assume $|Z\cap A|\geq\frac{1}{2}|Z|$. 
 Set $Z'=Z\cap A$.  Then
\begin{multline*}
\sum_{(x,z)\in X\times Z'}\bone_A(zx)=\sum_{z\in Z'}|X\cap z\inv A|=\sum_{z\in Z'}|zX\cap A|\\
=\sum_{z\in Z'}(|X|-|zX\backslash A|)\leq (1-\epsilon)|X||Z'|,
\end{multline*}
where the final inequality uses $Z'\seq Z$ (and the definition of $Z$).
On the other hand, we also have
\begin{multline*}
\sum_{(x,z)\in X\times Z'}\bone_A(zx)=\sum_{x\in X}|Z'\cap Ax\inv|\geq\sum_{x\in X}(|Z'|-|A\backslash Ax\inv|)\\
\geq (|Z'|-N)|X|\geq (1-2 N|Z|\inv)|X||Z'|,
\end{multline*}
where the first inequality uses $Z'\seq A$ (recall $Z'=Z\cap A$), the second inequality uses $X\seq\Stab_{N}^r(A)$, and the final inequality uses $|Z'|\geq\frac{1}{2}|Z|$.
 After combining the above inequalities we obtain
 $$
 (1-\epsilon)|X||Z'|\geq (1-2N|Z|^{-1})|X||Z'|.
 $$
Canceling $|X||Z'|$ and rearranging yields  $|Z|\leq 2N/\epsilon$, as desired. 
\end{proof}

\begin{remark}\label{rem:regularity}
Using a similar proof, one can establish the version of  Lemma \ref{lem:regularity} with $\ell$ and $r$ exchanged. (It is also a straightforward exercise to deduce directly from the statement of Lemma \ref{lem:regularity} via Proposition \ref{prop:Stab}.)
\end{remark}

The second main result of this section is Lemma \ref{lem:structure} below, which implies any set $A$ is well approximated by certain sets arising from stabilizers of $A$. As noted above, this result crucially depends on a delicate interplay between left and right stabilizers.

\begin{lemma}[structure]\label{lem:structure}
Fix a nonempty finite set $A\seq G$ and some $\epsilon\in (0,1)$. Let $X=\St^\ell_{\epsilon^2/162}(A)$ and fix $\nu\in (0,1)$ satisfying $\nu\leq |X|/|A|$. Set $S=\St^r_{\epsilon\nu/9}(A)$ and $A'=\{a\in A:|Xa\backslash A|<\frac{\epsilon}{9}|X|\}$. Suppose $D$ is any subset of $G$ satisfying  $A'\seq D\seq A'S$. Then $|A\smd D|<\epsilon|A|$. 
\end{lemma}
\begin{proof}
Let $\delta=\frac{\epsilon}{9}$. Then we have $X=\St^\ell_{\delta^2/2}(A)$, $S=\St^r_{\delta\nu}(A)$, and 
$$
A'=\{a\in A:|Xa\backslash A|<\delta|X|\}.
$$

\noindent\textit{Claim 1.} $|A'|\geq (1-3\delta)|A|$.

\noindent\textit{Proof.} Let $Z=Z^r_\delta(A,X)$. Then $|Z|\leq \delta|A|$ by Lemma \ref{lem:regularity} and Remark \ref{rem:regularity}.

Next observe that by definition, $A'\seq A\backslash Z$. Set $A''=(A\backslash Z)\backslash A'$, and note that if $a\in A''$ then  $|Xa\cap A|<\delta|X|$. We  now have
\begin{multline*}
\sum_{(a,x)\in (A\backslash Z)\times X}\bone_A(xa)=\sum_{x\in X}|(A\backslash Z)\cap x\inv A|\\
\geq \sum_{x\in X}(|A\backslash Z|-|A\backslash x\inv A|)\geq  (1-2\delta)|X||A|,
\end{multline*}
where the final inequality uses $|Z|\leq\delta|A|$ and  $|A\backslash x\inv A|\leq\frac{\delta^2}{2}|A|\leq\delta|A|$ (recall $X=\St^\ell_{\delta^2/2}(A)$). 
On the other hand,
\begin{multline*}
\sum_{(a,x)\in (A\backslash Z)\times X}\bone_A(xa)=\sum_{a\in A\backslash Z}|X\cap Aa\inv |=\sum_{a\in A'}|Xa\cap A|+\sum_{a\in A''}|Xa\cap A|\\
\leq |X||A'|+\delta|X||A''|\leq |X|(|A'|+\delta|A|).
\end{multline*}
Combining these inequalities, we obtain $(1-2\delta)|X||A|\leq |X|(|A'|+\delta |A|)$.  Canceling $|X|$ and rearranging yields $|A'|\geq (1-3\delta)|A|$, as desired.  \clqed\medskip

\noindent\textit{Claim 2.} $(1-2\delta)|A'S|\leq |A|$.

\noindent\textit{Proof.} We first fix $g\in A'S$ and show $(1-2\delta)|X|<|A \cap Xg|$. Write $g=as$ for some $a\in A'$ and $s\in S$. Then
\[
|Xg\backslash A|=|Xa\backslash As\inv|\leq |Xa\backslash A|+|A\backslash As\inv|.
\]
Since $a\in A'$, we have $|Xa\backslash A|<\delta|X|$. Also, since $s\in S=\St^r_{\delta\nu}(A)$, we have $|A\backslash As\inv|\leq \delta\nu|A|\leq \delta|X|$, where the last inequality is by our assumption $\nu\leq |X|/|A|$. So $|Xg\backslash A|<2\delta|X|$, i.e., $(1-2\delta)|X|<|A \cap Xg|$. 

Next, note that $B\coloneqq A'S\cup XA$ is finite. By the above,
\begin{multline*}
(1-2\delta)|X||A'S|=\sum_{g\in A'S}(1-2\delta)|X|\leq\sum_{g\in A'S}|A\cap Xg|\leq \sum_{g\in B}|A\cap Xg|\\
=\sum_{(g,x)\in B\times X}\bone_A(xg)=\sum_{x\in X}|B\cap x\inv A|=\sum_{x\in X}|x\inv A|=|X||A|,
\end{multline*}
where the second inequality uses the definition of $B$, and the second to last inequality uses the definition of $B$ and the fact $X$ is symmetric.  After canceling $|X|$, this yields $(1-2\delta)|A'S|\leq |A|$. \clqed\medskip

Finally, fix a set $D\seq G$ such that $A'\seq D\seq A'S$. By the two claims,
\begin{align*}
|A\smd D| &= |A|+|D|-2|A\cap D|\leq |A|+|A'S|-2|A'|\\
 &\leq \big(1+(1-2\delta)\inv-2(1-3\delta)\big)|A|\\
 &=\big(2\delta(1-2\delta)\inv+6\delta\big)|A|.
\end{align*}
Note that $2(1-2\delta)\inv< 3$ (since $\delta< \frac{1}{6}$). Consequently, we have shown that $|A\smd D|< 9\delta|A|=\epsilon|A|$.
\end{proof}

\section{A new proof of Theorem \ref{thm:CPTNIP}}\label{sec:ineff}

We now prove the first main result of the paper, which leads to a new proof of Theorem \ref{thm:CPTNIP} (see Remark \ref{rem:compare}). 

\begin{theorem}\label{thm:NIPAR}
Let $G$ be a finite group. Fix a nonempty  set $A\seq G$, and let $d=\max\{\VC^\ell_A(A),\VC^r_A(A)\}$ and $\alpha=|A|/|G|$. Then for any $\epsilon\in (0,1)$, there is a Bohr neighborhood $B\seq \St^r_\epsilon(A)$ of complexity $O_{d,\alpha,\epsilon}(1)$ satisfying the following properties:
\begin{enumerate}[$(i)$]
\item \textnormal{(structure)} There is a set $F\seq A$ with $|F|\leq O_{d,\alpha,\epsilon}(1)$ such that 
\[
|A\smd FB|<\epsilon|A|.
\]
\item \textnormal{(regularity)} $|Z^\ell_\epsilon(A,B)|\leq \epsilon|A|$. 
\end{enumerate}
\end{theorem}
\begin{proof}
Let $X=\St^\ell_{\epsilon^2/162}(A)$ and set $\nu=\min\{|X|/|A|,\epsilon\}$. By Corollary \ref{cor:HPL}$(a)$, we have $\cov(A:X)\leq O_{d,\alpha,\epsilon}(1)$, hence $\nu\inv\leq O_{d,\alpha,\epsilon}(1)$.\footnote{Since this proof will ultimately yield ineffective bounds, we have not stated the explicit bounds from Corollary \ref{cor:HPL} here for the sake of simplicity.}  Now set   $R=\St^r_{\epsilon\nu/36}(A)$. Then $\cov(A\inv:R)\leq O_{d,\alpha,\epsilon}(1)$ by Corollary \ref{cor:HPL}$(b)$. So $|A|\leq O_{d,\epsilon,\alpha}(|R|)$, and hence $|G|\leq O_{d,\alpha,\epsilon}(|R|)$. By Theorem \ref{thm:Bogo}, $R^4$ contains a Bohr neighborhood $B$ of complexity $O_{d,\alpha,\epsilon}(1)$. Thus $B\seq R^4\seq S\coloneqq\St^r_{\epsilon\nu/9}(A)$. Since $\epsilon\nu/9\leq \epsilon$, we have $B\seq\St^r_\epsilon(A)$. Since $\epsilon\nu/9\leq\epsilon^2/2$, condition $(ii)$ follows from  Lemma \ref{lem:regularity}.

For condition $(i)$, we first apply Lemma \ref{lem:structure} to obtain  $A'\seq A$ such that for any $D\seq G$, if $A'\seq D\seq A'S$ then $|A\smd D|<\epsilon|A|$. Using Fact \ref{fact:Bohr}$(b)$, let $C$ be a Bohr neighborhood  of complexity $O_{d,\alpha,\epsilon}(1)$ such that $C^2\seq B$. By Fact \ref{fact:Bohr}$(c)$, $|A'C|\leq |G|\leq O_{d,\alpha,\epsilon}(|C|)$. So we can apply Lemma \ref{lem:RCL} to find some $F\seq A'$ with $|F|\leq O_{d,\alpha,\epsilon}(1)$ such that $A'\seq FC^2$. Altogether
    \[
    A'\seq FC^2\seq FB\seq  A'S.
    \]
    Therefore $|A\smd FB|<\epsilon|A|$, and we have condition $(i)$. 
\end{proof}

In the next few remarks, we compare Theorem \ref{thm:NIPAR} to Theorem \ref{thm:CPTNIP} and other aspects of  \cite{CPTNIP}. 

\begin{remark}\label{rem:compare}
Theorem \ref{thm:CPTNIP} follows immediately from Theorem \ref{thm:NIPAR}, in light of two basic observations:
\begin{enumerate}[$(1)$]
\item The statement of Theorem \ref{thm:CPTNIP} is trivial if $|A|<\epsilon|G|$.
\item The subgroup $H$ in Theorem \ref{thm:CPTNIP} is embedded in our notion of Bohr neighborhood (recall Remark \ref{rem:Bohrdef}).
\end{enumerate}

Conversely,  by applying Theorem \ref{thm:CPTNIP} with $\epsilon\alpha$, we obtain a statement almost identical to  Theorem \ref{thm:NIPAR}, except that:
\begin{enumerate}[$(i)$]
\item we must set $d=\VC^\ell_G(A)$ (instead of $\max\{\VC^\ell_A(A),\VC^r_A(A)\}$), and
\item we do not have $B\seq \St^r_\epsilon(A)$.
\end{enumerate}
Regarding $(i)$, recall from Remark \ref{rem:VCvariants} that $\max\{\VC^\ell_A(A),\VC^r_A(A)\}$ can be bounded above uniformly in terms of $\VC^\ell_G(A)$, but do not know whether the converse holds. It could also be possible that $\max\{\VC^\ell_A(A),\VC^r_A(A)\}$ suffices in the context of \cite{CPTNIP}, but checking this in the underlying model-theoretic ingredients would require some effort. As for $(ii)$ however, even though this feature is not made explicit in \cite{CPTNIP}, it can be obtained from  methods used in the proof due to the relationship between model-theoretic connected components and stabilizers in the pseudofinite NIP setting (see \cite[Section 3]{CPpfNIP}). In fact, one can even arrange for $B$ to simultaneously be contained in $\St^\ell_\epsilon(A)$. The model-theoretic explanation of this can be found in \cite[Section 5]{CPfrNIP} (in a more general setting than that of \cite{CPTNIP}). We will revisit this in our setting in Remark \ref{rem:Sint} below.
\end{remark}

\begin{remark}\label{rem:NIPAR}
Here we discuss three differences between Theorem \ref{thm:CPTNIP} and the main result of \cite{CPTNIP} (which is \cite[Theorem 5.7]{CPTNIP}). 
\begin{enumerate}[$(1)$]
\item The result in \cite{CPTNIP} has a more precise bound on $|F|$ in terms of the complexity of $B$. We obtain the same bound here via Fact \ref{fact:Bohr}$(c)$ (modulo the change in metric discussed in Remark \ref{rem:Bohrdef}). 
\item In \cite{CPTNIP}, the error set $Z$ is  a Boolean combination  of bi-translates of $A$. The same is true of the subgroup $H$ associated to $B$ (via Remark \ref{rem:Bohrdef}). Our proof does not provide these features. It seems potentially possible to recover this for $Z$ using other tools from VC-theory (such as $\epsilon$-nets). The picture is less clear for $H$, which here is obtained via Theorem \ref{thm:Bogo}. The proofs  of this theorem in \cite{CoBogo} and \cite{CHP} freely expand the language, and thus lose control of definability. One can likely address this in the framework of \cite{CoBogo} using the Massicot-Wagner \cite{MassWa} treatment of Sanders' results in \cite{SanBS}, and obtain some kind of definability information for $H$. The proof of Theorem \ref{thm:Bogo} in \cite{CP-AVSAR} would provide similar information, but in a different setting based on continuous logic. 

\item The most significant difference between Theorem \ref{thm:CPTNIP} and \cite{CPTNIP} is that the proof of \cite[Theorem 5.7]{CPTNIP} provides ``functional control" of the error in the regularity statement. More precisely, in the conclusion of the regularity statement, one can replace $\epsilon|B|$ with $f(\delta,m)|B|$ where $f$ is a fixed function of the complexity parameters.\footnote{This is not explicitly stated in \cite[Theorem 5.7]{CPTNIP}, but is evident from the main lemma used in the proof (\cite[Lemma 5.6]{CPTNIP}). Details are provided in the supplemental note to \cite{CPTNIP} available on the first author's webpage.} This allows one to obtain the structure statement very easily from the regularity statement (compared to Lemma \ref{lem:structure}), and in fact yields a  stronger structure statement:  $|(A\smd FB)\backslash Z|<\epsilon|B|$. This kind of functional control on the error in regularity lemmas is a desirable feature. An application in the context of NIP arithmetic regularity is discussed in \cite[Section 10]{CPfrNIP}. Significant applications in the context of stable arithmetic regularity for functions are obtained in \cite{CP-AVSAR}. It would be very interesting to prove analogues of the lemmas in Section \ref{sec:lemmas} with functional control on the regularity error. 
\end{enumerate}
\end{remark}

\begin{remark}\label{rem:Sint}
Continuing along the lines of the end of Remark \ref{rem:compare}, we show that one can also obtain $B\seq \St^\ell_\epsilon(A)$ in the above proof of Theorem \ref{thm:NIPAR}. To see this, first note that the set $R$ can be replaced by any symmetric subset $R'\seq R$ satisfying $|A|\leq O_{d,\alpha,\epsilon}(|R'|)$. So let $R'=R\cap \St^\ell_{\epsilon/4}(A)$. To verify $|A|\leq O_{d,\alpha,\epsilon}(|R'|)$, first set $U=\St^r_{\epsilon\nu/72}(A)$ and $V=\St^\ell_{\epsilon/8}(A)$. Then $U^2\seq R$ and $V^2\seq \St^\ell_{\epsilon/4}(A)$, so $U^2\cap V^2\seq R'$. By a general  pigeonhole argument, we have $|U^2\cap V^2|\geq|U||V|/|UV|$ (for example, this is evident from the proof of \cite[Lemma 5.8]{BGT}). So $|U||V|\leq |UV||R'|$. Moreover, we  have $|UV|\leq\alpha\inv|A|$ by definition of $\alpha$, $|A|\leq O_{d,\alpha,\epsilon}(|U|)$ by Corollary \ref{cor:HPL}$(b)$, and $|A|\leq O_{d,\alpha,\epsilon}(|V|)$ by Corollary \ref{cor:HPL}$(a)$. This altogether yields $|A|\leq O_{d,\alpha,\epsilon}(|R'|)$. 
\end{remark}

\section{Effective results on NIP sets of bounded tripling}\label{sec:eff}

\subsection{A generalized Alon-Fox-Zhao trick}

In the proof of Theorem \ref{thm:NIPAR}, we lose the effective bounds given by Haussler's Packing Lemma  because of the application of the noncommutative Bogolyubov's Lemma (Theorem \ref{thm:Bogo}). A similar kind of situation arises in the work of Alon, Fox, and Zhao \cite{AFZ} on NIP sets in finite abelian groups of bounded exponent. In particular, a direct application of Bogolyubov's Lemma to the stabilizer of an NIP set would ruin the polynomial bounds given by Haussler. For this reason, Alox, Fox, and Zhao use a clever trick to pass to a large subset of the stabilizer whose doubling parameter can be controlled independently of $\epsilon$. The next lemma extracts the essence of this approach from the proof of \cite[Lemma 2.4]{AFZ}, though we have simplified things slightly since we do not aim for as sharp of a bound. We also move to the setting of arbitrary groups.  

\begin{lemma}\label{lem:AFZ}
Let $G$ be a  group and fix a nonempty finite set $A\seq G$. Let $d$, $k$, and $m$ be defined as either:
\begin{enumerate}
\item $d= \VC^\ell_A(A)$, $k=|A^2|/|A|$, and $m=|AA\inv|/|A|$, or
\item $d= \VC^\ell_{A\inv}(A)$, $k=|A\inv A|/|A|$, and $m=|AA\inv|/|A|$.
\end{enumerate}
Assume further that $d\geq 1$.
Then for any integers $n\geq u\geq 2$, and any $\epsilon\in (0,1)$, there is a symmetric set $B\seq G$ such that $|B^u|\leq u^{d(d+1)}|B|$, $B^n\seq \St^\ell_\epsilon(A)$, and $|A|\leq m^{1/d}(30kn/\epsilon)^{d+1}|B|$.
\end{lemma}
\begin{proof}
Set $\delta=m^{\nv 1/d^2}(30k)^{\nv1/d}(\epsilon/n)^{1+1/d}$. Let  $R=\St^\ell_\delta(A)$. Then in both cases $(1)$ and $(2)$, $|A|\leq (30k/\delta)^{d}|R|$ by Corollary \ref{cor:HPL}. So 
\begin{align*}
|AA\inv|\leq c|R|,\tag{$i$}
\end{align*}
where $c=m(30k/\delta)^{d}$. Let $w=d(d+1)$. By choice of the parameters, one can compute that $c^{1/w}n\delta=\epsilon$. This yields the following conclusion.
\begin{flalign*}
\text{For any $x>0$, if $x^w\leq c$ then $xn\delta\leq c^{1/w}n\delta=\epsilon$.} && \tag{$\dagger$}
\end{flalign*}

\noindent\textit{Claim.} There is an integer  $t\geq 0$ such that $u^{tw}\leq c$ and $|R^{u^{t+1}}|< u^w|R^{u^t}|$.

\noindent\textit{Proof.} Suppose not. Let $t_*$ be the largest integer $t$ such that $u^{tw}\leq c$. Then for all $0\leq t\leq t_*$, $u^w|R^{u^t}|\leq |R^{u^{t+1}}|$. We can use this inductively to show that for all $0\leq t\leq t_*$, $u^{(t+1)w}|R|\leq |R^{u^{t+1}}|$.  Hence
\begin{align*}
u^{(t_*+1)w}|R|\leq |R^{u^{t_*+1}}|.\tag{$ii$}
\end{align*}

 Note that $R^{u^{t_*+1}}\seq \St^\ell_{u^{t_*+1}\delta}(A)$, and also $u^{t_*+1}\delta\leq un\inv \epsilon$ by $(\dagger)$. Since $un\inv \epsilon<2$, it follows from Proposition \ref{prop:Stab}$(d)$ that $R^{u^{t_*+1}}\seq AA\inv$. Combining this with inequalities $(i)$ and $(ii)$, we have
\[
u^{(t_*+1)w}|AA\inv|\leq u^{(t_*+1)w}c|R|\leq c|R^{u^{t_*+1}}|\leq c|AA\inv|.
\]
 Thus $u^{(t_*+1)w}\leq c$, which contradicts the choice of $t_*$.\clqed\medskip

Fix $t$ as in the claim, and set $B=R^{u^t}$. So  $|B^u|\leq u^w|B|$. Since $u^{tw}\leq c$, we have $u^tn\delta\leq\epsilon$ by $(\dagger)$. Thus, by Proposition \ref{prop:Stab}$(b)$,
\[
B^n=R^{u^tn}\seq \St^\ell_{u^tn\delta}(A)\seq \St^\ell_{\epsilon}(A).
\]
Finally, recall that $|A|\leq (30k/\delta)^d|R|$, and note that  $|R|\leq |B|$. Therefore 
\[
|A|\leq (30k/\delta)^d|B|=m^{1/d}(30kn/\epsilon)^{d+1}|B|,
\]
where the final equality is a direct calculation using the definition of $\delta$.
\end{proof}

\begin{remark}\label{rem:d=0}
In the previous lemma, the assumption $d\geq 1$ is needed to make sense of terms involving  $1/d$. But in fact, a  version of the lemma exists for $d=0$ as well. In particular, if $d=0$ then (in either case) we may use Proposition \ref{prop:d=0} to assume  that $A$ is a right coset $Ha$ of some subgroup $H\leq G$, and then the key content of Lemma \ref{lem:AFZ} holds with $B=H$.

For similar reasons, all of the results we prove below will be  trivial when $\VC^{\bullet}_B(A)=0$ for  $\bullet\in\{\ell,r\}$ and $B\in\{A,A\inv\}$. So we will frequently assume this is not the case without any loss in generality. In order for  the statements of our results to  make sense in this case, we remark that finite subgroups are coset nilprogressions of complexity $0$. If the ambient group is abelian, then  finite subgroups are proper coset progressions of rank $0$.
\end{remark}

\subsection{Polynomial Bogolyubov-Ruzsa for NIP sets}\label{sec:PBR}

As a warm-up to how Lemma \ref{lem:AFZ} will be used in conjunction with Theorem \ref{thm:BGT} in our main results,  we record the following corollary, which can be viewed as a strong form of Theorem \ref{thm:BGT}  for finite NIP sets of bounded tripling in arbitrary (possibly nonabelian) groups, with polynomial bounds.

\begin{corollary}\label{cor:PBR}
Let $G$ be a  group. Suppose $A\seq G$ is a nonempty finite set, and let $d=\VC^\ell_A(A)$ and $k=|A^3|/|A|$. Then there is a coset nilprogression $P$ of rank and step $O(d^2)$, and in $O_d(1)$-normal form, such that $P\seq AA\inv$ and $\cov(A:P)\leq O_d(k^{d+O(1)})$. 
\end{corollary}
\begin{proof}
We may assume $d\geq 1$ (see Remark \ref{rem:d=0}).
 Let $w=d(d+1)$. Recall that $|AA\inv|/|A|\leq k^{O(1)}$ by the Pl\"{u}nnecke-Ruzsa inequalities (Proposition \ref{prop:PRI}). So we can apply case $(1)$ of Lemma \ref{lem:AFZ} with $\epsilon=1/2$, $u=3$, and  $n=n(w)\leq O_d(1)$  from Definition \ref{def:BGTn}. This yields a symmetric set $B\seq G$ such that $|A|\leq O_d(k^{d+O(1)}|B|)$, $|B^3|\leq 3^w|B|$, and $B^n\seq \St^\ell_{\epsilon}(A)\seq AA\inv$. By Theorem \ref{thm:BGT}, there is a coset nilprogression $P$ of rank and step $O(\log(3^w))= O(d^2)$, and in $O_d(1)$-normal form, such that $P\seq B^m$ and $|B|\leq O_d(|P|)$. So  $|A|\leq O_d(k^{d+O(1)}|P|)$. 
 
 Now apply Fact \ref{fact:coverQ}$(b)$ to find a symmetric set $Q$ such that $Q^2\seq P$ and $|P|\leq O_d(|Q|)$. Then, using the Pl\"{u}nnecke-Ruzsa inequalities (Proposition \ref{prop:PRI}), we have
 \[
 |AQ|\leq |AP|\leq |AAA\inv|\leq k^{O(1)}|A|\leq O_d(k^{d+O(1)}|P|)\leq O_d(k^{d+O(1)}|Q|).
 \]
Since $Q^2\seq P$, this yields $\cov(A:P)\leq O_d(k^{d+O(1)})$ by Lemma \ref{lem:RCL}. 
\end{proof}

The previous result yields a polynomial bound in the tripling constant, provided we view the VC-dimension $d$ as fixed. From this perspective, we then obtain qualitative improvements over Theorem \ref{thm:BGT}. In particular, the complexity of $P$ depends \emph{only} on $d$, and we find $P$ inside $AA\inv$ rather than a fourfold product of $A$ and $A\inv$ (e.g., $2A-2A$ in the additive situation). This latter feature is a hallmark property of  NIP sets. 

\begin{remark}
In the previous corollary, a bound on $|A^2|/|A|$ is sufficient to obtain the size estimate $|A|\leq O_d(k^{d+O(1)}|P|)$ via case $(1)$ of Lemma \ref{lem:AFZ} (and Fact \ref{fact:triplings}$(b)$). By arguing with case $(2)$ instead, one could obtain a similar conclusion using $\VC^\ell_{A\inv}(A)$ and $\max\{|AA\inv|/|A|,|AA\inv|/|A|\}$. Along the same lines, if one were to distinguish $k=|A^3|/|A|$ from $c=|A^2|/|A|$, then the proof of Corollary \ref{cor:PBR} yields $\cov(A:P)\leq O_d(c^{d}k^{O(1)})$. 
\end{remark}

\subsection{NIP sets of bounded tripling}

In this section, we prove an effective version of \cite[Theorem 2.1]{CPfrNIP}, which is a result of of  the first author and Pillay \cite{CPfrNIP} on finite NIP sets of bounded tripling in arbitrary (possibly infinite) groups. 
In analogy to the discussion of \cite{CPTNIP} in Remark \ref{rem:NIPAR}$(2,3)$, our proof does not yield certain definability aspects of \cite[Theorem 2.1]{CPfrNIP} nor ``functional control" on the error. On the other hand, the bounds in \cite[Theorem 2.1]{CPfrNIP} are ineffective, whereas here we obtain bounds that are polynomial in the tripling and error parameter. Moreover, our result allows for some variation in the VC-dimension and tripling parameters. In particular (using the notation of Definition \ref{def:TP}),  one can choose either $\VC^\ell_A(A)$, $\VC^r_A(A)$ and $\tau[A]$, or $\VC^\ell_{A\inv}(A)$, $\VC^r_{A\inv}(A)\}$, and $\delta'[A]\coloneqq \max\{\delta[A],\delta[A\inv]\}$. By Remark \ref{rem:VCvariants}, $\VC^\bullet_A(A)$ can be bounded above uniformly in terms of $\VC^{\circ}_{A\inv}(A)$, but we do not know whether the converse holds. Likewise, by Fact \ref{fact:triplings}, $\delta'[A]$ is bounded above uniformly in terms of $\tau[A]$, but the converse is not true and, in some sense, bounding $\delta'[A]$ is the weakest reasonable assumption. On the other hand, there is no uniform  bound between $\delta[A]$ and $\delta[A\inv]$  (see \cite[Example 4.4]{TaoPSE}), hence we must explicitly bound both.

\begin{theorem}\label{thm:trip}
Let $G$ be a group and fix a  nonempty finite set $A\seq G$. Let $d_\ell$, $d_r$, and $k$ be defined as either:
\begin{enumerate}
\item $d_\ell=\VC^\ell_A(A)$, $d_r=\VC^r_A(A)$, and $k=|A^3|/|A|$, or
\item $d_\ell=\VC^\ell_{A\inv}(A)$, $d_r=\VC^r_{A\inv}(A)$, and $k=\max\{|AA\inv|,|A\inv A|\}/|A|$.
\end{enumerate}
Then for any $\epsilon\in (0,1)$, there is a coset nilprogression $P$ of rank and step $O(d_r^2)$, and in $O_{d_r}(1)$-normal form, and an integer $N$ satisfying
\[
N\leq O_{d_r}(\exp(O(d_\ell d_r))(k/\epsilon)^{O(d_\ell d_r)}),
\]
such that the following properties hold.
\begin{enumerate}[$(i)$]
\item $P\seq \St^r_{\epsilon}(A)$ and $\cov(A:P)\leq N$.
\item $|A\smd FP|<\epsilon|A|$ for some $F\seq A$ with $|F|\leq N$. 
\item  $|Z^\ell_\epsilon(A,P)|\leq \epsilon|A|$.
\end{enumerate}
\end{theorem}
\begin{proof}
We may assume $d_r\geq 1$ (see Remark \ref{rem:d=0}).
Let $X=\St^\ell_{\epsilon^2/162}(A)$ and set $\nu=(\epsilon^2/4680k)^{d_\ell}$. By Corollary \ref{cor:HPL}$(a)$, we have $|X|\geq\nu|A|$. Set $w=d_r(d_r+1)$, $\delta=\epsilon\nu/9$,  and $S=\St^r_{\delta}(A)$. We want to use $S$ in the context of Lemma \ref{lem:AFZ}. To justify this, we make the following easy observations, which follow from Proposition \ref{prop:Stab}$(c)$, Proposition \ref{prop:VCinv}, Fact \ref{fact:triplings}, and the fact that inversion preserves cardinality:
\begin{enumerate}[\hspace{5pt}$\ast$]
\item $S=\St^\ell_{\delta}(A\inv)$.
\item In case (1), $d_r=\VC^\ell_{A\inv}(A\inv)$, $|A^{\nv 2}|/|A\inv|\leq k$, and $|A\inv A|/A\inv|\leq k^2$.
\item In case (2), $d_r=\VC^\ell_{A\inv}(A)$ and $k=\max\{|A\inv A|,|AA\inv|\}/|A\inv|$.
\end{enumerate}
Thus we can apply Lemma \ref{lem:AFZ} to $A\inv$ with $u=3$,  $\delta$ in place of $\epsilon$, and $n=n(w)\leq O_{d_r}(1)$  from Definition \ref{def:BGTn}. This yields a symmetric set $B\seq G$ such that $|B^3|\leq 3^w|B|$, $B^n\seq S$, and $|A|\leq k^{O(1)}(30kn/\delta)^{d_r+1}|B|$. 
So $|A|\leq O_{d_r}((k/\delta)^{O(d_r)})|B|$.\footnote{We abandon $d_r+O(1)$ and work instead with $O(d_r)$ due to later steps in the proof.}

By Theorem \ref{thm:BGT}$(b)$, there is a coset nilprogression $P$ of rank and step $O(d_r^2)$, and in $O_{d_r}(1)$-normal form, such that $P\seq B^n\seq S$ and $|B|\leq O_{d_r}(|P|)$.  We will show that $P$ satisfies the desired properties. 

First, one can compute that $(k/\delta)^{O(d_r)}\leq O_{d_r}(\exp(O(d_\ell d_r))(k/\epsilon)^{O(d_\ell d_r)})$, and thus the bounds obtained below are of the desired form.

For condition $(i)$, first note that $P\seq S\seq  \St^r_\epsilon(A)$.  Using Fact \ref{fact:coverQ}$(b)$, fix a symmetric set $Q$ such that $Q^2\seq P$ and $|P|\leq O_{d_r}(|Q|)$. \medskip

\noindent\textit{Claim.} $|AQ|\leq O_{d_r}((k/\delta)^{O(d_r)}|Q|)$.

\noindent\textit{Proof.} In case $(1)$, we are argue as in the proof of Corollary \ref{cor:PBR}:
\begin{multline*}
|AQ|\leq |AP\leq |AA\inv A|\leq k^{O(1)}|A|\leq O_{d_r}((k/\delta)^{O(d_r)}|B|)\\
\leq O_{d_r}((k/\delta)^{d_r+O(1)}|P|)\leq O_{d_r}((k/\delta)^{O(d_r)}|Q|).
\end{multline*}
So assume  we are in case $(2)$. Then  Corollary \ref{cor:HPL}$(b)$  yields $\cov(A:S)\leq (30k/\delta)^{d_r}$. So we may fix  $E\seq A$ such that $|E|\leq (30k/\delta)^{d_r}$ and $A\seq ES$. Then $AP\seq ES^2$. Since $S^2\seq A\inv A$, we then get
\[
|AP|\leq |ES^2|\leq (30k/\delta)^{d_r}|A\inv A|\leq (30k/\delta)^{d_r}k|A|\leq O_{d_r}((k/\delta)^{O(d_r)}(|B|),
\]
which then yields the claim as in case $(1)$. \clqed\medskip

By the claim and Lemma \ref{lem:RCL}, we have $\cov(A:Q^2)\leq O_{d_r}((k/\delta)^{O(d_r)})$. So $\cov(A:P)\leq O_{d_r}((k/\delta)^{O(d_r)})$.  
This finishes condition $(i)$. 

For condition $(ii)$, recall that $S=\St^r_{\epsilon\nu/9}(A)$.  By Lemma \ref{lem:structure} and choice of $\nu$, there is a set $A'\seq A$ such that for any $D\seq G$, if $A'\seq D\seq A'S$ then $|A\smd D|<\epsilon|A|$. W can now apply Lemma \ref{lem:RCL} again (via the claim) to find $F\seq A'$ such that $|F|\leq O_{d_r}((k/\delta)^{O(d_r)})$ and $A'\seq FP$.  So $FP\seq A'S$, and hence $|A\smd FP|<\epsilon|A|$.

Finally, since $P\seq S\seq \St^r_{\epsilon^2/2}(A)$, Lemma \ref{lem:regularity} yields condition $(iii)$. 
\end{proof}

\begin{remark}
Although case $(1)$  bounds $\tau[A]$, the proof only explicitly uses bounds on $\sigma[A]$ and $\alpha[A]$. But this yields a bound on $\tau[A]$ by Fact \ref{fact:triplings}.
\end{remark}

For the sake of completeness, we note that Theorem \ref{thm:trip} yields a version of Theorem \ref{thm:NIPAR}, with (mostly) effective bounds, but in terms of a coset nilprogression rather than a Bohr neighborhood.

\begin{corollary}\label{cor:CPTeff}
Let $G$ be a finite group. Fix a nonempty set $A\seq G$, and let $d_\ell=\VC^\ell_A(A)$, $d_r=\VC^r_A(A)$, and $\alpha=|A|/|G|$. Then for any $\epsilon\in (0,1)$, there is a coset nilprogression $P$ of rank and step $O(d_r^2)$, and in $O_{d_r}(1)$-normal form, and an integer $N$ satisfying
\[
N\leq O_{d_r}(\exp(O(d_\ell d_r))(\alpha\epsilon)^{\nv O(d_\ell d_r)})
\]
such that the following properties hold.
\begin{enumerate}[$(i)$]
\item $P\seq \St^r_{\epsilon}(A)$, $\cov(A:P)\leq N$, and $\cov(G:P)\leq N$.
\item $|A\smd FP|<\epsilon|A|$ for some $F\seq A$ with $|F|\leq N$. 
\item  $|Z^\ell_\epsilon(A,P)|\leq \epsilon|A|$.
\end{enumerate}
Moreover, $\cov(G:P)\leq N$.
\end{corollary}
\begin{proof}
Other than $\cov(G:P)\leq N$, this follows from  case (1) of Theorem \ref{thm:trip} in light of the trivial  bound $|A^3|/|A|\leq |G|/|A|=\alpha\inv$. To also obtain $\cov(G:P)\leq N$, note that in the proof we can use $|G|\leq\alpha\inv|A|$  to obtain a suitable bound on $|G|/|Q|$, and then apply Lemma \ref{lem:RCL}.
\end{proof}

\begin{remark}
Using the same ideas as in Remark \ref{rem:Sint}, one could obtain stronger forms of Corollary \ref{cor:PBR} and Theorem \ref{thm:trip} with $P\seq \St^\ell_\epsilon(A)\cap\St^r_\epsilon(A)$, but with slightly worse bounds. For Theorem \ref{thm:trip}, this argument only works in case $(1)$. In particular, using the same notation as in  Remark \ref{rem:Sint}, we need an upper bound on $|UV|/|A|$. But the most we know is $UV\seq A\inv A^2 A\inv$ (via Proposition \ref{prop:Stab}$(d)$),  and a bound on $|A\inv A^2 A\inv|/|A|$ leads to a bound on $\tau[A]$. More precisely, by an exercise similar to Fact \ref{fact:triplings}$(c)$, one can show that if $\beta[A]\coloneqq |A\inv A^2|/|A|$ then $\tau[A]\leq\beta[A]\sigma[A]\leq\beta[A]^2$.
\end{remark}

\subsection{The abelian case}\label{sec:abelian}

The leading $O_d$ constants in Corollary \ref{cor:PBR} and Theorem \ref{thm:trip} are ineffective due to the application of Breuillard-Green-Tao. For abelian groups however, we can obtain fully explicit bounds by instead using the Bogolyubov-Ruzsa Lemma. This change also allows us to replace coset nilprogressions by (commutative) coset progressions. Moreover,  by Fact \ref{fact:triplings}$(e)$, we can focus solely on bounded doubling. 

Before getting to these arguments, we first remark briefly that these results  overlap with Sisask's \cite{SisNIP} work in the abelian case (a detailed discussion can be found at the end of this subsection). We also set some notation. Given an abelian group $G$ and finite sets $A,X\seq G$, we write $\St_\epsilon(A)$ for  $\St^\ell_\epsilon(A)=\St^r_\epsilon(A)$ and $Z_\epsilon(A,X)$ for $Z^\ell_\epsilon(A,X)=Z^r_\epsilon(A,X)$. Similarly, given $A,B\seq G$, we write  $\VC_B(A)$ for $\VC^\ell_B(A)=\VC^r_B(A)$. 

We start by specializing Corollary \ref{cor:PBR} to abelian groups. 

\begin{corollary}\label{cor:PBRa}
Let $G$ be an abelian group and fix a nonempty finite set $A\seq G$. Let $d=\VC_A(A)$ and $k=|2A|/|A|$. Then there is a proper coset progression $P$ of rank $O(d^{12})$ such that $P\seq A-A$ and $\cov(A:P)\leq \exp(O(d^{14}))k^{d+O(1)}$. 
\end{corollary}
\begin{proof}
We follow the proof of Corollary \ref{cor:PBR}, but using  $n=4$ and $u=2$ in the application of Lemma \ref{lem:AFZ}. This yields a symmetric set $B\seq G$ such that $|A|\leq O(1)^{d}k^{d+O(1)}|B|$, $|2B|\leq 2^{d(d+1)}|B|$, and $4B\seq A-A$. Now apply Theorem \ref{thm:BRL} to $B$  obtain a proper coset progression $P$ of rank $O(d^{12})$ such that $P\seq 4B$ and $|B|\leq \exp(O(d^{14}))|P|$. This yields $|A|\leq \exp(O(d^{14}))k^{d+O(1)}|P|$. So $\cov(A:P)\leq \exp(O(d^{14}))k^{d+O(1)}$ by a similar argument involving Fact \ref{fact:coverQa}, Lemma \ref{lem:RCL}, and Proposition \ref{prop:PRI}$(b)$. 
\end{proof}

Next we prove the commutative analogue of Theorem \ref{thm:trip}.  In light of Fact \ref{fact:triplings}$(e)$, we will only state the version corresponding to case $(1)$. 

\begin{theorem}\label{thm:tripa}
Let $G$ be an abelian group and fix a  nonempty finite set $A\seq G$. Let $d=\VC_A(A)$ and $k=|2A|/|A|$. Then for any $\epsilon\in (0,1)$, there is a proper coset progression $P$ of rank $O(d^{12})$ and an integer $N$ satisfying
\[
N\leq \exp(O(d^{14}))(k/\epsilon^2)^{d^2+O(d)},
\]
such that the following properties hold.
\begin{enumerate}[$(i)$]
\item $P\seq \St_{\epsilon}(A)$ and $\cov(A:P)\leq N$.
\item $|A\smd (F+P)|<\epsilon|A|$ for some $F\seq A$ with $|F|\leq N$. 
\item  $|Z_\epsilon(A,P)|\leq \epsilon|A|$.
\end{enumerate}
\end{theorem}
\begin{proof}
Similar to the comparison between Corollaries \ref{cor:PBR} and \ref{cor:PBRa}, we follow the proof of Theorem \ref{thm:trip} (in case $(1)$), but using $n=4$ and $u=2$ in the application of Lemma \ref{lem:AFZ}. This yields a symmetric set $B\seq G$ such that $|2B|\leq 2^{d(d+1)}|B|$, $4B\seq S$, and $|A|\leq 120^{d+1}(k/\delta)^{d+O(1)}|B|$ (where $S$ and $\delta$ are as in Theorem \ref{thm:trip}). Now, as the proof of Corollary \ref{cor:PBRa},  apply Theorem \ref{thm:BRL} to find a proper coset progression $P$ of rank $O(d^{12})$ such that $P\seq 4B$ and $|B|\leq \exp(O(d^{14}))|P|$. So $P\seq S\seq \St_\epsilon(A)$. Using Fact \ref{fact:coverQa}, fix a symmetric set $Q$ such that $2Q\seq P$ and $|P|\leq 4^{O(d^{12})}|Q|$. 

The key claim is now
\[
|A+Q|\leq \exp(O(d^{14}))(k/\delta)^{d+O(1)}|Q|.
\]
Given this, we can follow the rest of the proof verbatim. In order to verify that the bounds obtained are of the right form, one must also check  
\[
\exp(O(d^{14}))(k/\delta)^{d+O(1)}\leq \exp(O(d^{14}))(k/\epsilon^2)^{d^2+O(d)}.
\]

To establish the claim, we argue as in  case $(1)$ of Theorem \ref{thm:trip}. In particular, we have $|A+Q|\leq k^{O(1)}|A|$ via Pl\"{u}nnecke-Ruzsa, and the claim then follows from  the given bounds on $|A|/|B|$, $|B|/|P|$, and $|P|/|Q|$.
\end{proof}

As in Corollary \ref{cor:CPTeff}, we can restrict the previous theorem to finite abelian groups with the parameters $d=\VC_A(A)$ and $\alpha=|A|/|G|$. This yields a bound of the form $\exp(O(d^{14}))(\alpha\epsilon^2)^{\nv(d^2+O(d))}$. 
In \cite{SisNIP}, Sisask proves a similar statement  with $P$  replaced by a $(\delta,m,1)$-Bohr neighborhood $B$ with $\delta\inv,m\leq O(d\log((\alpha\epsilon)\inv))$. A bound on $\cov(G:B)$ is not explicitly given, but via Fact \ref{fact:Bohr}$(c)$ one can obtain $N\leq (\alpha\epsilon)^{\nv O(d\log(d\log((\alpha\epsilon)\inv)))}$.  

Along the same lines, we note that in \cite[Section 10]{CPfrNIP}, the first author and Pillay used ``modeling lemmas" for abelian groups in order to prove an ineffective version of Theorem \ref{thm:tripa}  directly from the   results for finite groups in \cite{CPTNIP}. This strategy was motivated by the work of Sisask \cite{SisNIP}, where something similar is done for the special case of $\F_q$-vector spaces. In this situation,  coset progressions can be replaced by  subgroups, which are  preserved by Freiman isomorphism (a key ingredient in the modeling lemma strategy). This approach also preserves polynomial bounds thanks to a ``polynomial modeling lemma" for $\F_q$-vector spaces (see \cite[Proposition 6.1]{GrRuz}, \cite[Lemma 5.6]{SisNIP}). For general abelian groups however, there are two obstacles to this strategy. The first is that, while an effective modeling lemma for abelian groups does exist (due to Green and Ruzsa \cite[Proposition 1.2]{GrRuz}), its bounds cannot be made polynomial in general  (see \cite[Proposition 6.4]{GrRuz}). The second obstacle is that in order to apply  Freiman isomorphisms to Sisask's result for finite abelian groups, one must first obtain coset progressions  from Bohr neighborhoods. But this seems to requires a tighter control of the error  than what is made explicit in the statement of Sisask's result. On the other hand, ``functional control" of the error (discussed in Remark \ref{rem:NIPAR}) is sufficient for this purpose, and is also available in the results of \cite{CPTNIP}. However, those results are ineffective.

Our proof of Theorem \ref{thm:tripa} resolves the problem in a more direct way that avoids the need for polynomial modeling. That said, we do use Theorem \ref{thm:BRL}, which itself involves a weak modeling lemma (\cite[Lemma 2.1]{GrRuz}). Altogether, it is not unreasonable to expect that Sisask's methods could be suitably modified to give an even more direct  proof of Theorem \ref{thm:tripa} that also avoids the full power of Theorem \ref{thm:BRL}.

\section{Bounded Exponent}\label{sec:BE}

In this section, we refine Theorem \ref{thm:trip} in the bounded exponent setting. The main improvement is that in this situation, coset (nil)progressions can be replaced by subgroups. For the case of dense sets in finite abelian  groups of bounded exponent, or sets of bounded doubling in (possibly infinite) $\F_q$-vector spaces, the appropriate analogue of Theorem \ref{thm:trip} is already known from Alon, Fox, Zhao \cite{AFZ} and Sisask \cite{SisNIP}, respectively. An ineffective treatment of the general situation is given by the first author and Pillay in \cite{CPfrNIP}, with some effective results in the abelian case in \cite[Section 10]{CPfrNIP}. However, a general result with  polynomial bounds has not been done previously, even in the abelian case (with a bounded exponent assumption).\footnote{Continuing with the discussion at the end of Section \ref{sec:abelian}, Sisask's strategy for $\F_q$-vector spaces would easily generalize given a polynomial modeling lemma for abelian groups of bounded exponent. But the existence of such a result appears to be open.}

We begin with a bounded exponent analogue of Theorem \ref{thm:BGT}. Throughout this section, we say that a subset $X$ of a group $G$ \emph{has exponent $q$} if $x^q=1$ for all $x\in X$.

\begin{theorem}[\cite{BGT,HruAG,vdDag}]\label{thm:BGTexp}
Let $G$ be a group and fix a  finite symmetric set $S\seq G$ such that $|S^3|\leq k|S|$ and $S^6$ has exponent $q$. Then there is a subgroup $H$ of $G$ such that $H\seq S^{12}$ and $|S|\leq O_{k,q}(|H|)$. 
\end{theorem}

As with Theorem \ref{thm:BGT}, our statement of this result is slightly different from  primary sources. In particular, $S$ is typically assumed to be an approximate group and the bounded exponent assumption is often stated for all of $G$. In this form, the result was first proved by Hrushovski \cite[Corollary 4.18]{HruAG}, and later again by Breuillard, Green, and Tao \cite[Theorem 6.15]{BGT}. In \cite[Corollary 5.6]{vdDag}, van den Dries observes that it suffices to bound the exponent of a suitable power of $S$. Finally, one can use  \cite[Corollary 3.10]{TaoPSE} to replace approximate groups with sets of bounded tripling.

Using Theorem \ref{thm:BGTexp},  one can prove a bounded exponent analogue of Theorem \ref{thm:trip}  by following the proof almost exactly. However, since coset nilprogressions are replaced by subgroups, Lemma \ref{lem:structure} can be replaced by a mild generalization of the  ``Stabilizer Lemma" of Alon, Fox, and Zhao \cite{AFZ} (discussed in the introduction), which has a much simpler proof. For completeness, we include a sketch of this argument, which has been extracted from the proof of \cite[Lemma 2.4]{AFZ} and generalized to finite subsets of arbitrary groups (following the treatment  in \cite[Lemma 8.2]{CoBogo} for  finite groups). 

\begin{lemma}\label{lem:AFZH}
Let $G$ be a group and fix a nonempty finite set $A\seq G$. Suppose $H$ is a finite subgroup of $G$ such that $H\seq\Stab^r_N(A)$ for some $N>0$. Then there is a set $D\seq G$, which is a union of left cosets of $H$, such that $|A\smd D|\leq N$.
\end{lemma}
\begin{proof}[Proof (Sketch)]
Let $\cC$ be the (finite) set of left cosets of $H$ in $G$ whose intersection with $A$ is nonempty. For $C\in\cC$ set $P_C=(C\cap A)\times (C\backslash A)$. Set $P=\bigcup_{C\in\cC}P_C$. Then we have $P=\{(a,g)\in A\times G\backslash A:a\inv g\in H\}$. Using this and a basic sum-switching argument, one can show $|P|=\sum_{x\in H}|Ax\backslash A|$. Since $H\seq \Stab^r_N(A)$, this yields $2|P|=\sum_{x\in H}|Ax\smd A|\leq N|H|$.

Now set $D=\bigcup\{C\in\cC:|C\cap A|\geq|H|/2\}$. Then we have 
\[
|A\smd D|=\sum_{C\in\cC}\min\{|C\cap A|,|C\backslash A|\}\leq \sum_{C\in\cC}\frac{2|P_C|}{|H|}=\frac{2|P|}{|H|}\leq N,
\]
where the first equality follows from the choice of $D$, and the middle inequality follows from the fact that  $\min\{x,y\}\leq 2xy/(x+y)$ for any $x,y>0$.
\end{proof}

\begin{remark}
In \cite[Lemma 8.2]{CoBogo}, it is also observed that a regularity statement can  be quickly obtained from the proof of the previous lemma. Indeed, let $\mathcal{Z}=\{C\in\cC:|P_C|\geq\epsilon^2|H|^2\}$. Then 
\[
N|H|/2\geq |P|\geq\sum_{C\in\mathcal{Z}}|P_C|\geq\epsilon^2|H|^2|\mathcal{Z}|.
\]
Thus $|\mathcal{Z}|\leq N/(2\epsilon^2|H|)$. So if $Z=\bigcup\mathcal{Z}$ then $|Z|\leq N/2\epsilon^2$. One can easily check  $Z^\ell_\epsilon(A,H)\seq Z$. So we have $|Z^\ell_\epsilon(A,H)|\leq N/2\epsilon ^2$.

On the other hand, recall that Lemma \ref{lem:regularity} yields $|Z^\ell_\epsilon(A,H)|\leq 2N/\epsilon$, which is better (when $\epsilon\leq 1/4$). 
\end{remark}

Now we can prove the bounded exponent analogue of Theorem \ref{thm:trip}.

\begin{theorem}\label{thm:tripBE}
Let $G$ be a group and fix a  nonempty finite set $A\seq G$ such that $A\inv A$ has exponent $q$. Let $d$ and $k$ be defined as either:
\begin{enumerate}
\item $d=\VC^r_A(A)$ and $k=|A^3|/|A|$, or
\item $d=\VC^r_{A\inv}(A)$ and $k=\max\{|AA\inv|,|A\inv A|\}/|A|$.
\end{enumerate}
Then for any $\epsilon\in (0,1)$, there is subgroup $H\leq G$ satisfying the following properties:
\begin{enumerate}[$(i)$]
\item $H\seq \St^r_{\epsilon}(A)$ and $\cov(A:H)\leq O_{d,q}((k/\epsilon)^{2d+O(1)})$.
\item $|A\smd FH|\leq \epsilon|A|$ for some $F\seq A$ with $|F|\leq O_{d,q}((k/\epsilon)^{2d+O(1)})$. 
\item  $|Z^\ell_{\epsilon^{1/2}}(A,H)|\leq 2\epsilon^{1/2}|A|$.
\end{enumerate}
\end{theorem}
\begin{proof}
We may assume $d\geq 1$ (see Remark \ref{rem:d=0}).
Set $S=\St^r_\epsilon(A)$ and $w=d(d+1)$. By the same justification as in the proof of Theorem \ref{thm:trip}, we can apply Lemma \ref{lem:AFZ} viewing $S$ as $\St^\ell_\epsilon(A\inv)$, with parameters $\epsilon$, $u=3$, and $n=12$. This yields a symmetric set $B$ such that $|B^3|\leq 3^w|B|$, $B^{12}\seq S$, and $|A|\leq k^{d\inv}(360k/\epsilon)^{d+1}|B|$.
Note, in particular, that $B^6\seq S\seq A\inv A$ (by Proposition \ref{prop:Stab}$(d)$), hence $B^6$ has exponent $q$. By Theorem \ref{thm:BGTexp}, there is a subgroup $H\leq G$ such that $H\seq B^{12}\seq S$ and $|B|\leq O_{d,q}(|H|)$. Thus $|A|\leq O_{d,q}((k/\epsilon)^{d+2}|H|)$. This leads to the following main claim, whose proof is essentially identical to the claim in Theorem \ref{thm:trip}.\medskip

\noindent\textit{Claim.} Set $N=|AH|/|H|$. Then $N\leq O_{d,q}((k/\epsilon)^{2d+O(1)})$.\footnote{In case $(1)$, the proof actually yields $d+O(1)$ in the power on $(k/\epsilon)$.}\medskip

Since $H$ is a subgroup, we get $\cov(A:H)\leq N$ by the claim and Lemma \ref{lem:RCL}. Since $H\seq S$, this yields condition $(i)$. For condition $(ii)$, apply Lemma \ref{lem:AFZH}  to obtain some $F\seq G$ such that $|A\smd FH|\leq \epsilon|A|$. Since $\cov(A:H)\leq N$, we can change coset representatives if necessary in order to obtain $F\seq A$ and $|F|\leq N$. Finally, condition $(iii)$ follows from Lemma \ref{lem:regularity}. 
\end{proof}

As before, the $O_{d,q}$ constant in the previous result is ineffective for general groups, but can be made effective in the abelian case using the Bogolyubov-Ruzsa Lemma for abelian groups of bounded exponent. Here the best known bounds are also due to Sanders \cite[Theorem 11.1]{SanBR}. By modifying the proof of Theorem \ref{thm:tripBE} in analogy to how Theorem \ref{thm:tripa} modifies Theorem \ref{thm:trip}, this leads to the following result (we omit the argument). 

\begin{theorem}\label{thm:tripaBE}
Let $G$ be an abelian group of exponent $q$ and fix a  nonempty finite set $A\seq G$. Let $d=\VC_A(A)$ and $k=|2A|/|A|$. Then for any $\epsilon\in (0,1)$, there is subgroup $H\leq G$ and an integer $N\leq \exp(O_q(d^8))(k/\epsilon)^{d+O(1)}$  satisfying the following properties:
\begin{enumerate}[$(i)$]
\item $H\seq \St_{\epsilon}(A)$ and $\cov(A:H)\leq N$.
\item $|A\smd (F+H)|\leq\epsilon|A|$ for some $F\seq A$ with $|F|\leq N$. 
\item  $|Z_{\epsilon^{1/2}}(A,H)|\leq 2\epsilon^{1/2}|A|$.
\end{enumerate}
\end{theorem}

As a final remark, recall that Theorem \ref{thm:BGT} is not exactly analogous to Theorem \ref{thm:BRL} due to the appearance of $S^{O_k(1)}$ rather than $S^{O(1)}$. This connects to recent work of Gowers, Green, Manners, and Tao \cite{GGMT2, GGMT}, which establishes the Polynomial Freiman-Ruzsa Conjecture for abelian groups of bounded exponent. Their main result (\cite[Theorem 1.1]{GGMT}) also involves a sumset iterate depending on the doubling constant, and thus does not yield Polynomial Bogolyubov-Ruzsa for abelian groups of bounded exponent. Nevertheless, in analogy to how Theorem \ref{thm:BGT} is used in the proof of Theorem \ref{thm:trip}, it is natural to ask what happens to the dependence on $d$ in Theorem \ref{thm:tripaBE} if we use \cite[Theorem 1.1]{GGMT} instead of \cite[Theorem 11.1]{SanBR}. However, the improvement is rather modest, namely, $\exp(O_q(d^8))$ can be replaced by $\exp(O_q(d^2))$. (An exponential dependence on $d$ is unavoidable due to the form of Haussler's Packing Lemma.)

\appendix
\section{}

In this section, we elaborate on the assertions made in Remark \ref{rem:VCvariants}. The following is the main result we want to establish.

\begin{proposition}\label{prop:VCvariants}
Let $A$ be a subset of a group $G$.
\begin{enumerate}[$(a)$]
\item $\VC^\ell_{B}(A)\leq \VC^\ell_G(A)$ for any $B\seq G$.
\item $\VC^\ell_G(A)-1\leq \dim_{\lVC}(A)\leq \VC^\ell_G(A)$.
\item $\VC^\ell_G(A)$ is the dual $\VC$-dimension of $\cF^r_G(A)$; hence
\[
\VC^\ell_G(A)<2\exp(\VC^r_G(A))\makebox[.5in]{and} \VC^r_G(A)<2\exp(\VC^\ell_G(A)).
\]
\item $\dim_{\lVC}(A)$ is the dual $\VC$-dimension of $\cF^r_{A\inv}(A)$; hence
\[
\dim_{\lVC}(A)<2\exp(\VC^r_{A\inv}(A))\makebox[.5in]{and}\VC^r_{A\inv}(A)<2\exp(\dim_{\lVC}(A)).
\]
\item The above statements also hold with $r$ and $\ell$ exchanged.
\end{enumerate}
\end{proposition}

Recall that $\dim_{\lVC}$ and $\dim_{\rVC}$ are Sisask's variations from \cite{SisNIP} (see Definition \ref{def:SVC} below for details). Toward proving Proposition \ref{prop:VCvariants}, note first that part $(a)$ is obvious since $\cF^\ell_B(A)\seq\cF^\ell_G(A)$. Part $(b)$ is also relatively straightforward, and is proved by Sisask in \cite[Proposition 4.1]{SisNIP}. In order to establish the remaining claims,  we first need to recall VC duality. 

Let $\cF$ be a set system (on some set $X$). The  \textbf{dual system} $\cF^*$ is the set system on $\cF$ consisting of the sets $\cF_x\coloneqq\{S\in\cF:x\in S\}$ as $x$ ranges over $\bigcup\cF$.\footnote{Note that $\bigcup\cF$ may be a proper subset of $X$. Many sources instead define $\cF^*$ relative to the fixed  set $X$. This discrepancy is minor, as it only affects whether $\cF^*$ contains $\emptyset$.} The VC-dimension of $\cF^*$  is called the \textbf{dual $\VC$-dimension of $\cF$}, and denoted $\VC^*(\cF)$. The following is a standard fact (see, e.g., \cite{Assouad}).

\begin{fact}\label{fact:duals}
If $\cF$ is a set system then 
\[
\VC^*(\cF)<2\exp(\VC(\cF))\makebox[.5in]{and}\VC(\cF)<2\exp(\VC^*(\cF)).
\]
\end{fact}

The assertions in Proposition \ref{prop:VCvariants} equate the VC-dimension of one set system $\cF_1$ with the dual VC-dimension of another set system $\cF_2$. This arises from a certain equivalence between $\cF_1$ and $\cF^*_2$, which is made rigorous by the following definition. 

\begin{definition}\label{def:quotient}
Let $\cF_1$ be a set system on $X=\bigcup\cF_1$ and let $\cF_2$ be a set system on $Y=\bigcup\cF_2$. Then $\cF_2$ is a \textbf{quotient} of $\cF_1$ if there is a surjective function $\sigma\colon X\to Y$ such that $\cF_1=\sigma\inv(\cF_2)$.
\end{definition}

The following is a straightforward exercise (see \cite[Proposition 2.2]{Assouad}).

\begin{proposition}\label{prop:SSiso}
Let $\cF_1$ be a set system on $X=\bigcup\cF_1$ and let $\cF_2$ be a set system on $Y=\bigcup\cF_2$. If $\cF_2$ is a quotient of $\cF_1$, then $\VC(\cF_1)=\VC(\cF_2)$.
\end{proposition}

Another informative exercise is to show that if $\cF$ is a set system on $X=\bigcup\cF$ then $(\cF^*)^*$ is a quotient of $\cF$ (and hence $\VC^*(\cF^*)=\VC(\cF)$). Indeed, $(\cF^*)^*$ is the ``smallest" quotient of $\cF$ in the following sense. Define an equivalence relation $\sim$ on $X$ such that $x\sim y$ if and only if $\cS_x=\cS_y$. Then every set in $\cF$ is $\sim$-invariant, which yields a well-defined set system $\cF/\!\!\sim$ on $X/\!\!\sim$. One can then show that $(\cF^*)^*$ and $\cF/\!\!\sim$ are ``equal" in the sense that Definition \ref{def:quotient} holds with a \emph{bijective} function. More generally, if $\cF_2$ is a quotient of $\cF_1$, then $(\cF_1^*)^*$ and $(\cF_2^*)^*$ are equal in this same sense. This justifies the idea that set systems are equivalent to their quotients.

We can now return to groups. It will be convenient to recall Sisask's \cite{SisNIP} variation in its full generality.

\begin{definition}\label{def:SVC}
Let $G$ be a group and fix $A,B\seq G$. Define the set systems 
\[
\cF^\ell(A|B)=\{xA\cap B:x\in BA\inv\}\makebox[.5in]{and} \cF^r(A|B)=\{Ax\cap B:x\in A\inv B\}.
\]
For $\bullet\in\{\ell,r\}$, set 
\[
\dim_{\bullet\!\VC}(A|B)=\VC(\cF^\bullet(A|B))\makebox[.5in]{and} \dim_{\bullet\!\VC}(A)=\dim_{\bullet\!\VC}(A|A).
\]
(In \cite{SisNIP}, Sisask writes $\dim_{\VC}$ rather than $\dim_{\lVC}$; we add the $\ell$ for consistency with the rest of our notation.)
\end{definition}

To ease notation, given  $\bullet\in\{\ell,r\}$, we  let $\cF^{\bullet *}_B(A)=(\cF^\bullet_B(A))^*$ and  $\VC^{\bullet*}_B(A)=\VC^*(\cF^{\bullet}_B(A))$. 

\begin{proposition}
Let $G$ be a group and fix  $A,B\seq G$.
\begin{enumerate}[$(a)$]
\item $\cF^{r*}_{B\inv}(A)$ is a quotient of $\cF^\ell(A|B)$, and hence $\dim_{\lVC}(A|B)=\VC^{r*}_{B\inv}(A)$. In particular,  $\dim_{\lVC}(A)=\VC^{r*}_{A\inv}(A)$. 
\item $\cF^{r*}_G(A)$ is a quotient of $\cF^\ell_G(A)$, and hence $\VC^\ell_G(A)=\VC^{r*}_G(A)$. 
\end{enumerate}
Moreover, the same holds with $\ell$ and $r$ exchanged.
\end{proposition}
\begin{proof}
We first prove part $(a)$. To ease notation, let $\cF=\cF^r_{B\inv}(A)$. If at least one of $A$ or $B$ is empty, then $\cF^\ell(A|B)=\emptyset=\cF^*$. So assume $A$ and $B$ are nonempty. Then we have $\bigcup\cF=AB\inv$, $\bigcup\cF^\ell(A|B)=B$, and $\bigcup\cF^*=\cF$. Define $\sigma\colon B\to \cF$ so that $\sigma(g)=Ag\inv$. Then for $x\in AB\inv$, we have $\sigma\inv(\cF_x)=x\inv A\cap B$, and thus $\sigma\inv(\cF^*)=\cF(A|B)$. Since $\sigma$ is surjective, this shows that $ \cF^*$ is a quotient of $\cF(A|B)$. Proposition \ref{prop:SSiso} then yields $\dim_{\lVC}(A|B)=\VC^{r*}_{B\inv}(A)$. 

For part $(b)$, note that $\cF^\ell(A|G)=\cF^\ell_G(A)$, and so this is a special case of part $(a)$. (Part $(b)$ is also a well-known fact; see, e.g., \cite[p.13]{CPTNIP}, \cite[Corollary 3.19$(b)$]{CPfrNIP}. One can further check that $\cF^{r\ast}_G(A)$ is ``equal" to $\cF^{\ell\ast\ast}_G(A)$ in the sense discussed after Proposition \ref{prop:SSiso}.) 

The proof with $\ell$ and $r$ exchanged is similar (or apply $(a)$ to $A\inv,B\inv$).
\end{proof}

Together with Fact \ref{fact:duals}, the previous proposition establishes the remaining results in Proposition \ref{prop:VCvariants}.


\begin{thebibliography}{10}

\bibitem{AFZ}
N. Alon, J. Fox, and Y. Zhao, \emph{Efficient arithmetic regularity and removal
  lemmas for induced bipartite patterns}, Discrete Anal. (2019), Paper No. 3,
  14. \MR{3943117}

\bibitem{Assouad}
P. Assouad, \emph{Densit\'{e} et dimension}, Ann. Inst. Fourier (Grenoble)
  \textbf{33} (1983), no.~3, 233--282. \MR{723955}

\bibitem{Bog39}
N. Bogolio\`uboff, \emph{Sur quelques propri\'et\'es arithm\'etiques des
  presque-p\'eriodes}, Ann. Chaire Phys. Math. Kiev \textbf{4} (1939),
  185--205. \MR{0020164}

\bibitem{BreuH5P}
E. Breuillard, \emph{Lectures on approximate groups and {H}ilbert's 5th
  problem}, Recent trends in combinatorics, IMA Vol. Math. Appl., vol. 159,
  Springer, [Cham], 2016, pp.~369--404. \MR{3526417}

\bibitem{BGT}
E. Breuillard, B. Green, and T. Tao, \emph{The structure of approximate
  groups}, Publ. Math. Inst. Hautes \'Etudes Sci. \textbf{116} (2012),
  115--221. \MR{3090256}






\bibitem{CoBogo}
G. Conant, \emph{On finite sets of small tripling or small alternation in
  arbitrary groups}, Combin. Probab. Comput. \textbf{29} (2020), no.~6,
  807--829. \MR{4173133}

\bibitem{CoQSAR}
\bysame, \emph{Quantitative structure of stable sets in arbitrary finite
  groups}, Proc. Amer. Math. Soc. \textbf{149} (2021), no.~9, 4015--4028.
  \MR{4291597}


\bibitem{CGH2}
G. Conant, K. Gannon, and J. Hanson, \emph{Generic stability, randomizations,
  and {NIP} formulas}, arXiv:2308.01801, 2023.


\bibitem{CHP}
G. Conant, E. Hrushovski, and A. Pillay, \emph{Compactifications of
  pseudofinite and pseudo-amenable groups}, Groups Geom. Dyn. (to appear).

  
\bibitem{CPpfNIP}
G. Conant and A. Pillay, \emph{Pseudofinite groups and {VC}-dimension}, J.
  Math. Log. \textbf{21} (2021), no.~2, Paper No. 2150009, 23. \MR{4290498}

\bibitem{CPfrNIP}
\bysame, \emph{Approximate subgroups with bounded {VC}-dimension}, Math. Ann.
  \textbf{388} (2024), no.~1, 1001--1043. \MR{4693953}

\bibitem{CP-AVSAR}
\bysame, \emph{An analytic version of stable arithmetic
  regularity}, arXiv:2401.14363, 2024.


\bibitem{CPT}
G. Conant, A. Pillay, and C. Terry, \emph{A group version of stable
  regularity}, Math. Proc. Cambridge Philos. Soc. \textbf{168} (2020), no.~2,
  405--413. \MR{4064112}

\bibitem{CPTNIP}
\bysame, \emph{Structure and regularity for subsets
  of groups with finite {VC}-dimension}, J. Eur. Math. Soc. (JEMS) \textbf{24}
  (2022), no.~2, 583--621. \MR{4382479}



\bibitem{vdDag}
L. van~den Dries, \emph{Approximate groups [according to {H}rushovski and
  {B}reuillard, {G}reen, {T}ao]}, Ast\'erisque (2015), no.~367-368, Exp. No.
  1077, vii, 79--113. \MR{3363589}


\bibitem{GGMT2}
W.~T. Gowers, B. Green, F. Manners, and T. Tao, \emph{On a conjecture of
  {M}arton}, arXiv:2311.05762, 2023.

\bibitem{GGMT}
\bysame, \emph{Marton's {C}onjecture in abelian groups with bounded torsion},
  arXiv:2404.02244, 2024.


\bibitem{GreenSLAG}
B. Green, \emph{A {S}zemer\'edi-type regularity lemma in abelian groups, with
  applications}, Geom. Funct. Anal. \textbf{15} (2005), no.~2, 340--376.
  \MR{2153903}

\bibitem{GrRuz}
B. Green and I.~Z. Ruzsa, \emph{Freiman's theorem in an arbitrary abelian
  group}, J. Lond. Math. Soc. (2) \textbf{75} (2007), no.~1, 163--175.
  \MR{2302736}


\bibitem{HaussPL}
D. Haussler, \emph{Sphere packing numbers for subsets of the {B}oolean
  {$n$}-cube with bounded {V}apnik-{C}hervonenkis dimension}, J. Combin. Theory
  Ser. A \textbf{69} (1995), no.~2, 217--232. \MR{1313896}

\bibitem{HruAG}
E. Hrushovski, \emph{Stable group theory and approximate subgroups}, J. Amer.
  Math. Soc. \textbf{25} (2012), no.~1, 189--243. \MR{2833482}

\bibitem{HPP}
E. Hrushovski, Y. Peterzil, and A. Pillay, \emph{Groups, measures, and the
  {NIP}}, J. Amer. Math. Soc. \textbf{21} (2008), no.~2, 563--596. \MR{2373360
  (2008k:03078)}

\bibitem{LovSzeg}
L. Lov\'asz and B. Szegedy, \emph{Regularity partitions and the topology of
  graphons}, An irregular mind, Bolyai Soc. Math. Stud., vol.~21, J\'anos
  Bolyai Math. Soc., Budapest, 2010, pp.~415--446. \MR{2815610}

\bibitem{Lov-surv}
S. Lovett, \emph{An exposition of {S}anders' quasi-polynomial {F}reiman-{R}uzsa
  theorem}, Theory of Computing (2015), 1--14.

\bibitem{LovReg}
S. Lovett and O. Regev, \emph{A counterexample to a strong variant of the
  polynomial {F}reiman-{R}uzsa conjecture in {E}uclidean space}, Discrete Anal.
  (2017), Paper No. 8, 6. \MR{3651924}


\bibitem{MaShStab}
M. Malliaris and S. Shelah, \emph{Regularity lemmas for stable graphs}, Trans.
  Amer. Math. Soc. \textbf{366} (2014), no.~3, 1551--1585. \MR{3145742}



\bibitem{MassWa}
J.-C. Massicot and F.~O. Wagner, \emph{Approximate subgroups}, J. \'Ec.
  polytech. Math. \textbf{2} (2015), 55--64. \MR{3345797}

\bibitem{Moran-Yeh}
S. Moran and A. Yehudayoff, \emph{On weak {$\epsilon$}-nets and the {R}adon
  number}, Discrete Comput. Geom. \textbf{64} (2020), no.~4, 1125--1140.
  \MR{4183358}

\bibitem{PiRCP}
A. Pillay, \emph{Remarks on compactifications of pseudofinite groups}, Fund.
  Math. \textbf{236} (2017), no.~2, 193--200. \MR{3591278}

\bibitem{Ruz94}
I.~Z. Ruzsa, \emph{Generalized arithmetical progressions and sumsets}, Acta
  Math. Hungar. \textbf{65} (1994), no.~4, 379--388. \MR{1281447}

\bibitem{RuzBE}
\bysame, \emph{An analog of {F}reiman's theorem in groups}, Ast\'{e}risque
  (1999), no.~258, 323--326. \MR{1701207}



\bibitem{SanBS}
T. Sanders, \emph{On a nonabelian {B}alog-{S}zemer\'edi-type lemma}, J. Aust.
  Math. Soc. \textbf{89} (2010), no.~1, 127--132. \MR{2727067}

\bibitem{SanBR}
\bysame, \emph{On the {B}ogolyubov-{R}uzsa lemma}, Anal. PDE \textbf{5}
  (2012), no.~3, 627--655. \MR{2994508}

\bibitem{SimRC}
P. Simon, \emph{Rosenthal compacta and {NIP} formulas}, Fund. Math.
  \textbf{231} (2015), no.~1, 81--92. \MR{3361236}

\bibitem{SimGCD}
\bysame, \emph{V{C}-sets and generic compact domination}, Israel J. Math.
  \textbf{218} (2017), no.~1, 27--41. \MR{3625123}


\bibitem{SisNIP}
O. Sisask, \emph{Convolutions of sets with bounded {VC}-dimension are uniformly
  continuous}, Discrete Anal. (2021), Paper No. 1, 25. \MR{4237082}


\bibitem{SzemRL}
E. Szemer\'edi, \emph{Regular partitions of graphs}, Probl\`emes combinatoires
  et th\'eorie des graphes ({C}olloq. {I}nternat. {CNRS}, {U}niv. {O}rsay,
  {O}rsay, 1976), Colloq. Internat. CNRS, vol. 260, CNRS, Paris, 1978,
  pp.~399--401. \MR{540024}

\bibitem{TaoPSE}
T. Tao, \emph{Product set estimates for non-commutative groups}, Combinatorica
  \textbf{28} (2008), no.~5, 547--594. \MR{2501249}

\bibitem{TaoVu}
T. Tao and V. Vu, \emph{Additive combinatorics}, Cambridge Studies in Advanced
  Mathematics, vol. 105, Cambridge University Press, Cambridge, 2006.
  \MR{2289012}


\bibitem{TeWo}
C. Terry and J. Wolf, \emph{Stable arithmetic regularity in the finite field
  model}, Bull. Lond. Math. Soc. \textbf{51} (2019), no.~1, 70--88.
  \MR{3919562}

\bibitem{TeWo2}
\bysame, \emph{Quantitative structure of stable sets in finite abelian groups},
  Trans. Amer. Math. Soc. \textbf{373} (2020), no.~6, 3885--3903. \MR{4105513}


\bibitem{Tointon}
M.~C.~H. Tointon, \emph{Freiman's theorem in an arbitrary nilpotent group},
  Proc. Lond. Math. Soc. (3) \textbf{109} (2014), no.~2, 318--352. \MR{3254927}

\end{thebibliography}
\end{document}